\documentclass[11pt,reqno]{amsart}
\usepackage[all]{xy}
\usepackage{amssymb}
\usepackage{amsthm}
\usepackage[normalem]{ulem}
\usepackage{amsmath,mathtools}
\usepackage{amscd,enumitem}
\usepackage{verbatim}
\usepackage{eurosym}
\usepackage{float}
\usepackage{color}
\usepackage{dcolumn}
\usepackage[mathscr]{eucal}
\usepackage[all]{xy}
\usepackage{bbm}
\usepackage[textheight=8.5in, textwidth=6.7in]{geometry}
\newtheorem*{conj*}{Conjecture}
\newtheorem{theorem}{Theorem}[section]

\theoremstyle{definition}
\newtheorem*{remark}{Remark}
\theoremstyle{plain}

\newtheorem{lemma}[theorem]{Lemma}

\newtheorem{prop}[theorem]{Proposition}

\newtheorem*{theorem*}{Theorem}

\newtheorem{corollary}[theorem]{Corollary}

\newcommand{\Z}{\mathbb{Z}}

\newcommand{\R}{\mathbb{R}}
\newcommand{\N}{\mathbb{N}}

\newcommand{\SL}{\operatorname{SL}}
\newcommand{\C}{\mathbb{C}}
\newcommand{\im}[1]{\text{Im}\(#1\)}
\newcommand{\re}[1]{\text{Re}\(#1\)}

\renewcommand{\pmod}[1]{\,\,({\rm mod}\,\,{#1})}

\DeclareMathOperator{\pl}{PL}

\usepackage{hyperref}
\hypersetup{
		colorlinks=true,
		citecolor=magenta,
		linkcolor=blue,
		filecolor=blue,      
		urlcolor=blue,
}

\numberwithin{equation}{section}

\usepackage{array}
\newcolumntype{L}{>{$}p{12cm}<{$}}
\newcolumntype{T}{>{$}l<{$}}

\newtheoremstyle{example}
  {\topsep}   
  {\topsep}   
  {\normalfont}  
  {0pt}       
  {\bfseries} 
  {.}         
  {5pt plus 1pt minus 1pt} 
  {}          
\theoremstyle{example}

\def\({\left(}
\def\){\right)}

\usepackage{centernot}

\begin{document}
\title{Log-concavity for partitions without sequences}
\author{Lukas Mauth}
\address{Department of Mathematics and Computer Science\\Division of Mathematics\\University of Cologne\\ Weyertal 86-90 \\ 50931 Cologne \\Germany}
\email{lmauth@uni-koeln.de}

\makeatletter
\@namedef{subjclassname@2020}{%
	\textup{2020} Mathematics Subject Classification}
\makeatother

\subjclass[2020]{11B57, 11F03, 11F20, 11F30, 11F37, 11P82}
\keywords{Circle Method, $\eta$-function, partitions}

\begin{abstract} 
	We prove log-concavity for the function counting partitions without sequences. We use an exact formula for a mixed-mock modular form of weight zero, explicit estimates on modified Kloosterman sums and analytic techniques. Finally, we establish the higher Turán inequalities in an asymptotic form of the aforementioned partition function using a well established criterion of Griffin, Ono, Rolen, and Zagier on the zeros of Jensen polynomials.
\end{abstract}

\maketitle
\section{Introduction and statement of results}

A {\it partition} of a positive integer $n$ is a non-increasing sequence of positive integers which sum to $n.$ We denote the number of partitions of $n$ by $p(n)$ and agree on the convention $p(0) \coloneqq 1$. A classical tool for studying partitions is the generating function \cite{Andrews}

\begin{equation}
P(q) := \sum_{n=0}^{\infty} p(n)q^n = \prod_{n=1}^{\infty} \frac{1}{1-q^n} = \frac{1}{(q;q)_{\infty}},
\end{equation}

\noindent
where for $a \in \C$ and $n \in \N_{0} \cup \{\infty\}$ we define the $q$-Pochhammer symbol $(a)_n= (a;q)_n := \prod_{j=0}^{n-1} (1-aq^j).$

This form of the generating function was used by Hardy and Ramanujan to develop the Circle Method and led them to prove the following asymptotics \cite{HardyRamanujan}

\begin{equation}\label{partitionasymptotics}
p(n) \sim \frac{1}{4n\sqrt{3}}e^{\pi\sqrt{\frac{2n}{3}}}, \quad n \rightarrow \infty.
\end{equation}

\noindent
Later, Rademacher improved the Circle Method to obtain an exact formula for $p(n)$ \cite{RademacherExact}. To state this formula, define the {\it Kloostermann sum} \cite{RademacherExact}

$$A_k(n) := \sum_{\substack{0 \leq h < k \\ \gcd(h,k)=1}} \omega_{h,k}e^{\frac{-2\pi i n h}{k}},$$

\noindent
where $\omega_{h,k}$ is a $24k$-th root of unity defined by

$$\omega_{h,k} := \begin{cases}
\left(\frac{-k}{h}\right) e^{-\pi i\left(\frac{1}{4}(2-hk-h) + \frac{1}{12}\left(k-\frac{1}{k}\right)\left(2h-h'+h^2h'\right)\right)} & \text{ if } h \text{ is odd,}  \\

\left(\frac{-h}{k}\right) e^{-\pi i\left(\frac{1}{4}(k-1) + \frac{1}{12}\left(k-\frac{1}{k}\right)\left(2h-h'+h^2h'\right)\right)} & \text{ if } k \text{ is odd,} 

\end{cases}$$

\noindent
where $h'$ is a solution to $hh' \equiv -1 \pmod k,$ and $\left(\frac{\cdot}{\cdot}\right)$ denotes the Jacobi symbol. Furthermore, let $I_\kappa$ denote the modified Bessel function of order $\kappa.$ Rademacher's exact formula for $p(n)$ can then be stated as (see \cite{RademacherZuckerman})

$$p(n) = \frac{2\pi}{(24n-1)^{3/4}} \sum_{k=1}^{\infty} \frac{A_k(n)}{k} I_{\frac{3}{2}}\left(\frac{\pi\sqrt{24n-1}}{6k}\right).$$

We note briefly that Rademacher wrote this formula down slightly different in his original work \cite{RademacherExact}. The first term in the sum recovers the asymptotics found by Hardy and Ramanujan. Rademacher's proof made extensive use of the fact that the generating function $P(q)$ is essentially a modular form. Zuckerman generalized Rademacher's method and showed that there are exact formulas for the Fourier coefficients (at any cusp) of all weakly holomorphic modular forms of {\it negative} weight of any finite index subgroup of $\SL_2(\Z)$ \cite{Zuckerman}. Rademacher and Zuckerman did the case where the subgroup is $\SL_2(\Z)$ itself together in \cite{RademacherZuckerman}.

From a combinatorial point of view it is a classical question to ask whether $p(n)$ is log-concave (see \cite{DesalvoPak, NicolasPartitions, OnoPujahariRolen, Stanley}), where we call a sequence $\{\alpha(n)\}$ of positive real numbers {\it log-concave} if it satisfies for all $n\geq 1$ the inequality 

$$\alpha(n)^2 \geq \alpha(n-1)\alpha(n+1).$$

Many important sequences that arise naturally in combinatorics are known to be log-concave, among them are binomial coefficients, Stirling numbers and Bessel numbers \cite{Stanley}. The first proof that $p(n)$ is log-concave for $n\geq 26$ and all even $n<26$ was given by Nicolas \cite{NicolasPartitions} and was later reproved independently by Desalvo and Pak \cite{DesalvoPak}. Ono, Pujahari and Rolen showed in \cite{OnoPujahariRolen} that the function counting plane partitions (see Section $4$ for a definition) is log-concave for all $n\geq 12.$ All these proofs follow the same principle. Using the Circle Method  (or an exact formula obtained by the Circle Method) one obtains strong asymptotics with explicit error bounds. The main term is easily seen to be log-concave and the difficulty lies in finding analytically a small bound after which the main term dominates the error term in the log-concavity condition. The remaining case are then checked directly. 

In this paper we intend to carry out this program for the function $p_2(n)$ which is defined as the number of partitions of $n$ that do not contain any consecutive integers as parts. These types of partitions were studied by Major MacMahon \cite{MacMahon} and arise in certain probability models and the study of threshold growth in cellular automata \cite{AndrewsErikssonPetrovRomik,HolroydLiggettRomik}. Andrews showed \cite{Andrews2} the following formula for the generating function:

$$G_2(q) := \sum_{n=0}^{\infty} p_2(n)q^n = \frac{(-q^3;q^3)_{\infty}}{(q^2;q^2)_{\infty}}\chi(q),$$

\noindent
where $\chi(q)$ denotes the third order mock theta function (see \cite{Andrews2})

$$\chi(q) := \sum_{n=0}^{\infty} \frac{(-q;q)_n}{(-q^3;q^3)_n}q^{n^2}.$$

In contrast to $p(n)$ the generating function of $p_2(n)$ is not anymore a modular form but the product of an mock theta function and an (essentially) modular form with an overall weight of $0.$ This increases difficulty of proving asymptotics for $p_2(n).$ Bringmann and Mahlburg first succeeded in proving an asymptotic formula for $p_2(n)$ \cite{BringmannMahlburg} and recent work \cite{BridgesBringmann} by Bridges and Bringmann improves their result and gives an exact formula for $p_2(n),$ and is the first example of an exact formula for the Fourier coefficients of a mixed-mock modular form. To state their result we need to introduce some notation. For $b > 0, k \in \N,$ and $\nu \in \N_{\geq 0}$ we define following \cite{BridgesBringmann}

$$\mathcal{I}_{b,k,\nu}(n) := \int_{-1}^{1} \frac{\sqrt{1-x^2}I_1\left(\frac{2\pi}{k}\sqrt{2bn(1-x^2)}\right)}{\cosh\left(\frac{\pi i}{k}\left(\nu - \frac{1}{6}\right)-\frac{\pi}{k}\sqrt{\frac{b}{3}}x\right)} dx.$$ 

Moreover, for $\nu \in \N_{\geq 0}$, and  for all $k, n  \in \N$ we define following \cite[Equations (3.1), (3.2), and (3.3)]{BridgesBringmann}\footnote{The reference has a typo in equation (3.3).} the Kloosterman sums

$$K_k^{[4]} (\nu;n) := \sum_{\substack{0 \leq h < k \\ \gcd(h,k) = 1 \\ 3 | h'}} \frac{\omega_{h,k}\omega_{h,\frac{k}{2}}\omega_{3h,k}}{\omega_{3h,\frac{k}{2}}} e^{\frac{\pi i}{k}\left(-3\nu^2+v\right)h'}e^{-\frac{2\pi i n h}{k}}, \quad \gcd(k,6)=2,$$

$$K_k^{[6]}(\nu; n) := \sum_{\substack{0 \leq h < k \\ \gcd(h,k) = 1 \\ 8 | h'}} \frac{\omega_{h,k}\omega_{2h,k}\omega_{h,\frac{k}{3}}}{\omega_{2h,\frac{k}{3}}} e^{\frac{\pi i}{k}\left(-3\nu^2+v\right)h'}e^{-\frac{2\pi i n h}{k}}, \quad \gcd(k,6)=3,$$

$$K_k^{[8]}(\nu;n) := \sum_{\substack{0 \leq h < k \\ \gcd(h,k) = 1 \\ 24 | h'}} \frac{\omega_{h,k}\omega_{2h,k}\omega_{3h,k}}{\omega_{6h,k}}e^{\frac{\pi i}{k}\left(-3\nu^2-v\right)h'}e^{\frac{-2\pi i n h}{k}}, \quad \gcd(k,6)=1,$$

$$\mathcal{K}_k(n) := \sum_{\substack{0 \leq h < k \\ (h,k) = 1}}\frac{\omega_{h,k}\omega_{2h,k}\omega_{6h,k}}{\omega_{3h,k}^3} e^{-\frac{2\pi i n h}{k}}, \quad \gcd(k,6)=1.$$

The exact formula for $p_2(n)$ shown by Bridges and Bringmann in  \cite{BridgesBringmann} then reads as follows.

\begin{theorem*} For $n \geq 1$ we have
	\begin{align}\label{exactformula}
	p_2(n) &= \frac{\pi}{6\sqrt{n}} \sum_{\substack{k \geq 1 \\ \gcd(k,6)=1}} \frac{\mathcal{K}_k(n)}{k^2} I_1\left(\frac{2\pi\sqrt{n}}{3k}\right) \nonumber \\
	&+\frac{\pi}{18\sqrt{6n}} \sum_{\substack{k \geq 1 \\ \gcd(k,6)=1}} \frac{1}{k^2} \sum_{\nu \pmod k} (-1)^{\nu} K_k^{[8]}(\nu;n) \mathcal{I}_{\frac{1}{18},k,\nu}(n) \nonumber \\
	&+\frac{5\pi}{36\sqrt{6n}} \sum_{\substack{k \geq 1 \\ \gcd(k,6)=2}} \frac{1}{k^2} \sum_{\nu \pmod k} (-1)^{\nu} K_k^{[4]}(v;n) \mathcal{I}_{\frac{5}{36},k,\nu}(n) \nonumber\\
	& + \frac{\pi}{6\sqrt{6n}} \sum_{\substack{k \geq 1 \\ \gcd(k,6)=3}} \frac{1}{k^2} \sum_{\nu \pmod k} (-1)^{\nu} K_k^{[6]}(\nu;n)\mathcal{I}_{\frac{1}{6},k,\nu}(n).
	\end{align}
\end{theorem*}

We are going to use this formula to obtain strong asymptotics for $p_2(n)$ in terms of elementary functions with an explicit error term. To carry out this program we need explicit estimates for the Kloosterman sums and the integrals $\mathcal{I}_{b,k,v}.$ We will derive these estimates in Section $2.$ It turns out that we need strong asymptotics for $\mathcal{I}_{\frac{1}{18},1,0}(n).$ This will be done in Section $3.$ In Section $4$ we are going to prove our main theorem

\begin{theorem*}
	We have for $n \geq 482$  and all even $2\leq n < 482$ that
	
	$$p_2^2(n) - p_2(n-1)p_2(n+1) > 0.$$
\end{theorem*}

Beyond log-concavity there are the {\it higher Turán inequalities}. For a non-vanishing sequence of real numbers $\{\alpha(n)\}$ and any positive integers $d$ and $n$ we define (following for instance \cite{GriffinOnoRolenZagier, OnoPujahariRolen}) the {\it Jensen polynomial of degree $d$ and shift $n$ associated to $\alpha$} by
\[
	J_{\alpha}^{d,n}(X) \coloneqq \sum_{j=0}^{d} \binom{d}{j}\alpha(n+j)X^j.
\]
We say that $\alpha(n)$ satisfies the {\it degree $d$ Turán inequality at $n$} \cite{OnoPujahariRolen} if the Jensen polynomial $J_{\alpha}^{d,n}$ is {\it hyperbolic} \cite{GriffinOnoRolenZagier}, i.e. has only real roots. It follows immediately that $\alpha(n)$ is log-concave $n$ if $J_{\alpha}^{d,n}$ has only real roots. Griffin, Ono, Rolen, and Zagier established a criterion \cite{GriffinOnoRolenZagier} to check whether for fixed $d \geq 1$ the Jensen polynomials $J_{\alpha}^{d,n}$ become eventually all hyperbolic if $n$ is large enough. This criterion can applied to a large class of sequences arising naturally when studying partitions. For concrete applications of the criterion the interested reader can check \cite{GriffinOnoRolenZagier, OnoPujahariRolen}.

In Section $5$ we will show that $p_2(n)$ satisfies for any fixed $d \geq 1$ the degree $d$ Turán inequalities for all but finitely many $n $ by using the before mentioned criterion.
\section*{Acknowledgements}
The author wishes to thank Kathrin Bringmann and Walter Bridges for suggesting and supervising this project, William Craig, Johann Franke, Caner Nazaroglu, and Badri Vishal Pandey for helpful discussions concerning log-concavity problems. Furthermore, the author wishes to thank William Craig, Ben Kane, Andreas Mono, Joshua Males, Caner Nazaroglu, and Matthias Storzer for helping me verify some of the results with a computer (see code\footnote{The code provided is written for the open-source mathematical software system SageMath \cite{Sage}.} attached to submission). Finally, the author thanks Koustav Banerjee for careful reading of the article and the helpful suggestions for improvement. The author recieved funding from the European Research Council (ERC) under the European Union’s Horizon 2020 research and innovation programme (grant agreement No. 101001179).

\section{Estimates for Kloosterman sums and Integrals}
First we are going to estimates the terms that come from $k\geq2$ in \eqref{exactformula} as they turn out to be exponentially smaller than the first term. We consider two different ranges, depending on whether $k$ is small or large relative to $n$. We say that $k$ is small if $k < 2\pi \sqrt{n}$ and we say that $k$ is large if $k\geq 2 \pi \sqrt{n}.$ We state a handful of premliminary results. 

First we need explicit upper bounds for the Bessel function $I_1(x)$ for small and large $x.$ From the following result \cite[Lemma 2.2]{BringmannKaneRolenTripp} due to Bringmann, Kane, Rolen and Tripp we conclude immediately the following upper bounds on the Bessel function $I_1(x)$.

\begin{lemma}\label{BesselFunctionUpperBounds}
	The following are true:
	\begin{align*}
	I_1(x) &\leq e^x, \quad x \geq 1, \\
	I_1(x) &\leq x, \quad 0\leq x < 1.
	\end{align*}
\end{lemma}

We want to write the modified Kloosterman sums in terms of classical ones. To that end, recall the definition of the {\it classical Kloosterman sum}. For $a,b \in \Z$ and $k\in \N$ we define

$$K(a,b,k) = \sum_{h \pmod k ^*} e^{\frac{2\pi i}{k}(ah+b[h]_k)},$$

\noindent
where $[h]_k$ denotes the inverse of $h$ mod $k.$ For classical Kloosterman sums we have the Weil bound, see \cite{IwaniecKowalski} for a proof. Let $\tau(n)$ denote the number of divisors of $n.$ 

\begin{theorem}[Weil bound for Kloosterman sums]\label{WeilBound}
	For $a,b \in \Z$ and $k \in \N$ we have 
	
	$$|K(a,b,k)| \leq \tau(k)\sqrt{\gcd(a,b,k)}k^{\frac{1}{2}}.$$ 
\end{theorem}

Furthermore, it is well known that $\tau(n) \ll_{\varepsilon} n^{\varepsilon}.$ The constant, depending on $\varepsilon$ can be worked out explicitly. We take the following result from \cite{Ramanujan}.

\begin{lemma}\label{DivisorBound}
	We have for all $n \geq 1$ that
	
	$$\tau(n) \leq 576\left(\frac{n}{21621600}\right)^{\frac{1}{4}} \leq 9 n^{\frac{1}{4}}.$$
\end{lemma}

We also need the following result from \cite[Lemma 3.4, Lemma 3.6, and Lemma 3.8]{BridgesBringmann} on the multipliers of the Kloosterman sums $K_k^{[4]}, K_k^{[6]}$ and $K_k^{[8]}$ to write them in terms of classical Kloosterman sums. 

\begin{lemma}\label{KloostermanSumMultipliers}
	We have 
	\begin{align*}
	\frac{\omega_{h,k}\omega_{h,\frac{k}{2}}\omega_{3h,k}}{\omega_{3h,\frac{k}{2}}} &= e^{\frac{2\pi i}{k}\left(\frac{k(k+2)}{8}h-\frac{k^2+2}{18}h'\right)},\\
	\frac{\omega_{h,k}\omega_{2h,k}\omega_{h,\frac{k}{3}}}{\omega_{2h,\frac{k}{3}}} &= (-1)^{\frac{k+1}{2}}e^{\frac{4\pi i k h}{9}-\frac{\pi i}{12k}(k^2-3)h'},\\
	\frac{\omega_{h,k}\omega_{2h,k}\omega_{3h,k}}{\omega_{6h,k}} &= (-1)^{\frac{k+1}{2}}e^{\frac{5\pi i (k^2-1)h'}{36k}}.
	\end{align*}
\end{lemma}

We are now ready to estimate the modified Kloosterman sums appearing in \eqref{exactformula}. This is an explicit version of the estimates given in \cite{BridgesBringmann} and follows their proof closely.

\begin{lemma}\label{KloostermanSumUpperBounds}
	For all $k\geq 1$ and $0 \leq \nu < k,$ we have 
	\begin{align*}
	\left|K_k^{[4]}(\nu;n)\right| &\leq 26\sqrt{n}k^{\frac{3}{4}}, \quad \gcd(k,6)=2,\\
	\left|K_k^{[6]}(\nu;n)\right| &\leq 27\sqrt{n}k^{\frac{3}{4}}, \quad \gcd(k,6)=3,\\
	\left|K_k^{[8]}(\nu;n)\right| &\leq 9 \sqrt{n} k^{\frac{3}{4}}, \quad \gcd(k,6) = 1,\\
	|\mathcal{K}_k(n)| &\leq k.
	\end{align*}
\end{lemma}

\begin{proof}
	Suppose $\gcd(k,6)=2.$ We then write using Lemma \ref{KloostermanSumMultipliers}
	
	$$K_k^{[4]}(\nu;n) = \sum_{\substack{0 \leq h < k \\ \gcd(h,k) = 1 \\ 3 | h'}} e^{\frac{2\pi i}{k}}\left(\left(\frac{k(k+2)}{8}-n\right)h+\left(-\frac{k^2+2}{18}+\frac{-3\nu^2+\nu}{2}\right)h'\right).$$
	
	We now change variables $h' \mapsto -3h'$ and $h \mapsto [-1]_k[3]_kh$ to obtain 
	
	$$K_k^{[4]}(\nu;n) = K\left([-1]_k[3]_k\left(\frac{k(k+2)}{8}-n\right),\frac{k^2+2}{6}+\frac{9\nu^2-3\nu}{2},k\right)$$
	
	Using Weil's bound, Theorem \ref{WeilBound}, we estimate this by
	
	$$K_k^{[4]}(\nu;n)  \leq \tau(k)\sqrt{\gcd\left([-1]_k[3]_k\left(\frac{k(k+2)}{8}-n\right),\frac{k^2+2}{6}+\frac{9\nu^2-3\nu}{2},k\right)}k^{\frac{1}{2}}.$$
	
	We continue to estimate 
	
	$$\gcd\left([-1]_k[3]_k\left(\frac{k(k+2)}{8}-n\right),\frac{k^2+2}{6}+\frac{9\nu^2-3\nu}{2},k\right) \leq \gcd\left([-1]_k[3]_k\left(\frac{k(k+2)}{8}-n\right),k\right)$$
	
	$$\leq \gcd\left([-1]_k[3]_k8n,k\right) = \gcd(8n,k) \leq 8n.$$
	
	Combined with the Lemma \ref{DivisorBound} this yields as claimed
	
	$$|K_k^{[4]}(\nu;n)| \leq 9k^{\frac{1}{4}}\sqrt{8n}k^{\frac{1}{2}} \leq 26\sqrt{n}k^{\frac{3}{4}}.$$
	
	The estimates for $K_k^{[6]}(\nu;n)$ and $K_k^{[8]}(\nu;n)$ follow in the same way. The estimate for $\mathcal{K}_k(n)$ is the trivial bound.
\end{proof}

In the next step we rewrite the integrals $\mathcal{I}_{b,k,\nu}(n).$ We start with the following elementary result.

\begin{lemma}\label{IntegrandSimplification}
	Let $k \in N,$ $\nu \in \N_0$ satisfying $0 \leq \nu \leq k$ and set $a\coloneqq a_{k,\nu} \coloneqq \frac{\pi}{k}(\nu-\frac{1}{6})$. Furthermore, let $b \geq 0$ and $x \in [0,1].$ Then, 
	
	$$f_{a,b}(x):=\frac{1}{\cosh(ai+bx)}+\frac{1}{\cosh(ai-bx)} = 4\frac{\cos(a)\cosh(bx)}{\cos(2a)+\cosh(2bx)}.$$
	
	\noindent
	In particular, $f_{a,b}$ is a real valued function with constant sign on $[0,1].$ Furthermore, $|f_{a,b}(x)|$ is monotonically decreasing on $[0,1].$
\end{lemma}

\begin{proof}
	We want to mention that there are no problems with poles. Indeed, $\cosh(z)=0$ if and only if $z =\frac{\pi}{2}(2n+1)i$ for $n \in \N_0$. One verifies that $a\in[-\frac{\pi}{6},\pi].$ This implies $z \in \{\frac{\pi}{2}, \frac{3\pi}{2}\}$. If we assume that $a=\frac{\pi}{2}$ then we must have $k= 2\nu - \frac{1}{3}$ which is absurd, since $k$ and $\nu$ are integers. In the other case where $a=\frac{3\pi}{2}$ we obtain $3k=2\nu - 1$. This gives again a contradiction since $0 \leq \nu \leq k$. Therefore, there are indeed no poles for the left hand side. Note that after the given equality in the statement is established, this is perfectly in line with the right hand side. If we look at the denominator $\cos(2a)+\cosh(2bx)$, then $\cosh(2bx) > 1$ unless $x=0$. Since $|\cos(x)| \leq 1$ the denominator precisely vanishes, when at $a = \frac{\pi}{2}(2n+1)$ and $x=0$.
	
	The rest of the proof is a straight forward calculus argument. A similar result in spirit, arising from an earlier study of partitions without sequences can be found in \cite[Proposition 5.1]{BringmannMahlburg}\footnote{Note that the reference contains typos.}.
\end{proof}

\begin{corollary}\label{BesselIntegralBound}
	Let $b$ be positive, $k\in \N$ and $\nu \in \N_{0}.$ We set $a_{k,\nu}:= \frac{\pi}{k}\left(\nu-\frac{1}{6}\right)$. Then, we have the estimate
	
	$$|\mathcal{I}_{b,k,\nu}(n)| \leq |2\sec(a_{k,\nu})|I_1\left(\frac{2\pi}{k}\sqrt{2bn}\right).$$
\end{corollary}

\begin{proof}
	Wes set $b_k:= \frac{\pi}{k}\sqrt{\frac{b}{3}}$ and use the previous lemma to write
	
	$$\mathcal{I}_{b,k,\nu}(n) = \int_{0}^{1} f_{a_{k,\nu},b_k}(x)\sqrt{1-x^2}I_1\left(\frac{2\pi}{k}\sqrt{2bn(1-x^2)}\right) dx.$$
	
	\noindent
	Recall that $|f_{a_{k,\nu},b_k}(x)|$ is decreasing. It is not hard to see that $I_1(x)$ is monotone increasing for positive $x$. Indeed, it follows immediately by considering the derivative of the series defining $I_1(x)$. Finally, Lemma \ref{IntegrandSimplification} shows that the maximum of the integrand in absolute value is at $x=0$ and we verify $f_{a_{k,\nu},b_k}(0) = 2\sec(a_{k,\nu}).$ The desired upper bound now follows.
\end{proof}

We state a last technical standard elementary result.

\begin{lemma}\label{lemma:bound-cos-sum}
	Let $k \in \N$ and $k \geq 2$. Then, we have the bound
	\[
		\sum_{\nu = 1}^{k} \left|\frac{1}{\cos\left(\frac{\pi}{k}\left(\nu-\frac{1}{6}\right)\right)}\right| \leq 8k\log(k).
	\]
\end{lemma}

\begin{proof}
	We consider for $x\in[0,\pi]$ the function
	\[
		g(x) = \frac{x-\frac{\pi}{2}}{\cos(x)}.
	\]
	We note that the $g(x)$ is well-defined since
	\[
		\lim_{x \rightarrow \frac{\pi}{2}} \frac{\cos(x)}{x-\frac{\pi}{2}} = \lim_{x \rightarrow 0} \frac{\cos(x+ \frac{\pi}{2})}{x} = - \lim_{x \rightarrow 0}  \frac{\sin x}{x} = -1.
	\]
	A standard argument now shows that $g(x) < 0$ for all $0 \leq x \leq \pi$ and that the minimal values are $g(0)=g(\pi)=-\frac{\pi}{2}$. From this it follows that for all $0 \leq x \leq \pi$ we have the inequality
	\begin{equation}\label{eq:cos-inequality}
	\left|\frac{1}{\cos(x)}\right| \leq 	\frac{\pi}{2}\frac{1}{\left|x-\frac{\pi}{2}\right|}.
	\end{equation}
	
	From \eqref{eq:cos-inequality} we infer
	\[
		\sum_{\nu = 1}^{k} \left|\frac{1}{\cos\left(\frac{\pi}{k}\left(\nu-\frac{1}{6}\right)\right)}\right| \leq \frac{\pi}{2}\sum_{\nu = 1}^{k} \frac{1}{\left|\frac{\pi}{k}\left(\nu-\frac{1}{6}\right) - \frac{\pi}{2}\right|} = k \sum_{\nu = 1}^{k} \frac{1}{\left|2\nu - \frac{1}{3}- k\right|}
	\]
	We split the last sum into two sums and bound them independently.
	\[
	\sum_{\nu = 1}^{\lfloor \frac{k}{2} \rfloor} \frac{1}{\left|2\nu - \frac{1}{3}- k\right|} = \sum_{\nu = 1}^{\lfloor \frac{k}{2} \rfloor} \frac{1}{k - 2\nu + \frac{1}{3}} \leq \sum_{\ell = 0}^{k} \frac{1}{\ell + \frac{1}{3}} \leq 3 + \sum_{\ell = 1}^{k} \frac{1}{\ell} \leq 3 +2\log(k).
	\]
	\[
		\sum_{\nu = \lfloor \frac{k}{2} \rfloor + 1}^{k} \frac{1}{\left|2\nu - \frac{1}{3}- k\right|} = \sum_{\nu = \lfloor \frac{k}{2} \rfloor + 1}^{k} \frac{1}{2\nu - \frac{1}{3}- k} \leq \sum_{\ell = 1}^{k} \frac{1}{\ell - \frac{1}{3}} \leq \frac{3}{2} + \sum_{\ell = 1}^{k} \frac{1}{\ell} \leq \frac{3}{2} + 2\log(k).
	\]
	Combining these estimates yields
	\[
		\sum_{\nu = 1}^{k} \left|\frac{1}{\cos\left(\frac{\pi}{k}\left(\nu-\frac{1}{6}\right)\right)}\right| \leq k \sum_{\nu = 1}^{k} \frac{1}{\left|2\nu - \frac{1}{3}- k\right|} \leq \frac{9}{2}k + 4k\log(k) \leq 8k\log(k),
	\]
	where we used in the last step that $\frac{9}{2} \leq 4\log(k)$ for $k \geq 7$. The cases $k=2,3,4,5,6$ are checked directly.
\end{proof}

\subsection{Estimates for large k}
In this subsection we always assume that $k \geq 2 \pi \sqrt{n}.$ 

\begin{lemma}\label{EstimateForLargeK}
	Let $n\geq 1$ and let $0 < b \leq \frac{1}{2}$. Then,
	
	$$\sum_{\nu = 1}^{k} |\mathcal{I}_{b,k,\nu}(n)| \leq 32\pi \log(k)\sqrt{2bn}.$$
\end{lemma}

\begin{proof} By Corollary \ref{BesselIntegralBound} and Lemma \ref{BesselFunctionUpperBounds} $\big($here we need $0<b\leq \frac{1}{2}\big)$ we find that 
	\[
		|\mathcal{I}_{b,k,\nu}(n)| \leq 2\left|\sec\left(\frac{\pi}{k}\left(\nu-\frac{1}{6}\right)\right)\right|\frac{2\pi}{k}\sqrt{2bn}.
	\]
	The claim now follows directly by summing these inequalities over $\nu = 1, \dots, k$ and applying Lemma \ref{lemma:bound-cos-sum}.
\end{proof}

This allows us to estimate the contribution from large $k$ to the sum in \eqref{exactformula}.

\begin{lemma}\label{UpperBoundLargeK}
	Let $n\geq 1.$ Then, we have 
	
	\begin{align}
	&\frac{\pi}{6\sqrt{n}} \left|\sum_{\substack{k \geq 2\pi\sqrt{n} \\ \gcd(k,6)=1}} \frac{\mathcal{K}_k(n)}{k^2} I_1\left(\frac{2\pi\sqrt{n}}{3k}\right)\right| \nonumber \\
	&+\frac{\pi}{18\sqrt{6n}} \left|\sum_{\substack{k \geq 2\pi\sqrt{n} \\ \gcd(k,6)=1}} \frac{1}{k^2} \sum_{\nu \pmod k} (-1)^{\nu} K_k^{[8]}(\nu;n) \mathcal{I}_{\frac{1}{18},k,\nu}(n) \right| \nonumber \\
	&+\frac{5\pi}{36\sqrt{6n}} \left|\sum_{\substack{k \geq 2\pi\sqrt{n} \\ \gcd(k,6)=2}} \frac{1}{k^2} \sum_{\nu \pmod k} (-1)^{\nu} K_k^{[4]}(v;n) \mathcal{I}_{\frac{5}{36},k,\nu}(n)\right| \nonumber\\
	& + \frac{\pi}{6\sqrt{6n}} \left|\sum_{\substack{k \geq 2\pi\sqrt{n} \\ \gcd(k,6)=3}} \frac{1}{k^2} \sum_{\nu \pmod k} (-1)^{\nu} K_k^{[6]}(\nu;n)\mathcal{I}_{\frac{1}{6},k,\nu}(n)\right| \leq 46500 n^{\frac{15}{16}}.
	\end{align}
\end{lemma}

\begin{proof}
	We start with the first sum. By Lemma \ref{BesselFunctionUpperBounds}  and Lemma \ref{KloostermanSumUpperBounds} we see
	\[
	\sum_{\substack{k \geq 2\pi\sqrt{n} \\ \gcd(k,6)=1}} \frac{|\mathcal{K}_k(n)|}{k^2} \left|I_1\left(\frac{2\pi\sqrt{n}}{3k}\right)\right|\leq \frac{2\pi}{3}\sqrt{n} \sum_{k \geq \sqrt{n}} \frac{1}{k^2}\leq \frac{2\pi}{3}\sqrt{n}\left( \frac{1}{n} + \int_{\sqrt{n}}^{\infty} \frac{1}{x^2} dx\right)=\frac{4\pi}{3\sqrt{n}}\leq 2\pi\leq 2\pi  n^{\frac{15}{16}}.
	\]
	Here we used for the second inequality the well-known upper bound for the sum against the integral for monotone decreasing functions.
	
	For the second sum we find using Lemma \ref{KloostermanSumUpperBounds} and Lemma \ref{EstimateForLargeK} that
	
	$$\sum_{\substack{k \geq 2\pi\sqrt{n} \\ \gcd(k,6)=1}} \frac{1}{k^2} \sum_{\nu = 1}^{k} \left|K_k^{[8]}(\nu;n)\right| \left|\mathcal{I}_{\frac{1}{18},k,\nu}(n)\right| \leq 9\sqrt{n}\sum_{k \geq 2\pi\sqrt{n}} \frac{1}{k^{\frac{5}{4}}} \sum_{\nu = 1}^{k} \left|\mathcal{I}_{\frac{1}{18},k,\nu}(n)\right|$$
	
	$$\leq 96\pi n \sum_{k \geq \sqrt{n}} \frac{\log(k)}{k^{\frac{5}{4}}}\leq 300\pi n \sum_{k \geq \sqrt{n}} \frac{1}{k^{\frac{9}{8}}} \leq 300\pi n \left( \frac{1}{n^{\frac{9}{16}}} + \int_{\sqrt{n}}^{\infty} \frac{1}{t^{\frac{9}{8}}} dt\right) \leq 2700\pi n^{\frac{15}{16}},$$
	
	\noindent
	where we used $\log(k) \leq 3k^{\frac{1}{8}}$ for all $n\geq 1.$ The other sums can be estimated similarly and we obtain the bounds
	
	$$\sum_{\substack{k \geq 2\pi\sqrt{n} \\ \gcd(k,6)=2}} \frac{1}{k^2} \sum_{\nu = 1}^{k} \left|K_k^{[4]}(v;n)\right| \left|\mathcal{I}_{\frac{5}{36},k,\nu}(n)\right| \leq 12000\pi n^{\frac{15}{16}},$$
	
	$$\sum_{\substack{k \geq 2\pi\sqrt{n} \\ \gcd(k,6)=3}} \frac{1}{k^2} \sum_{\nu = 1}^{k} \left|K_k^{[6]}(\nu;n)\right| \left|\mathcal{I}_{\frac{1}{6},k,\nu}(n)\right| \leq 13500\pi n^{\frac{15}{16}}.$$
	
	Furthermore, we have that 
	
	$$\max\left(\frac{5\pi}{36\sqrt{6n}},\frac{\pi}{6\sqrt{6n}},\frac{\pi}{6\sqrt{n}},\frac{\pi}{18\sqrt{6n}}\right) = \frac{\pi}{6\sqrt{n}}.$$
	
	Combining all estimates proves the claim.
\end{proof}

\subsection{Estimates for small k}

We assume throughout this subsection that $k < 2\pi\sqrt{n}.$ We start with the corresponding result to Lemma \ref{EstimateForLargeK} for small $k.$ 

\begin{lemma}
	Let $n \geq 1$ be fixed. Then we have
	
	\begin{align}
	&\frac{\pi}{6\sqrt{n}} \left|\sum_{\substack{2\leq k < 2\pi\sqrt{n} \\ \gcd(k,6)=1}} \frac{\mathcal{K}_k(n)}{k^2} I_1\left(\frac{2\pi\sqrt{n}}{3k}\right)\right| \nonumber \\
	&+\frac{\pi}{18\sqrt{6n}} \left|\sum_{\substack{2\leq k < 2\pi\sqrt{n} \\ \gcd(k,6)=1}} \frac{1}{k^2} \sum_{\nu \pmod k} (-1)^{\nu} K_k^{[8]}(\nu;n) \mathcal{I}_{\frac{1}{18},k,\nu}(n) \right| \nonumber \\
	&+\frac{5\pi}{36\sqrt{6n}} \left|\sum_{\substack{2\leq k < 2\pi\sqrt{n} \\ \gcd(k,6)=2}} \frac{1}{k^2} \sum_{\nu \pmod k} (-1)^{\nu} K_k^{[4]}(v;n) \mathcal{I}_{\frac{5}{36},k,\nu}(n)\right| \nonumber\\
	& + \frac{\pi}{6\sqrt{6n}} \left|\sum_{\substack{2\leq k < 2\pi\sqrt{n} \\ \gcd(k,6)=3}} \frac{1}{k^2} \sum_{\nu \pmod k} (-1)^{\nu} K_k^{[6]}(\nu;n)\mathcal{I}_{\frac{1}{6},k,\nu}(n)\right| \leq 200n^{\frac{1}{16}}e^{\frac{\sqrt{3}\pi}{3}\sqrt{n}}.
	\end{align}
\end{lemma}

\begin{proof}
	We can bound the first sum by Lemma \ref{BesselFunctionUpperBounds} and the trivial bound on $\mathcal{K}_k(n)$ as follows.
	\[
	\frac{\pi}{6\sqrt{n}} \sum_{\substack{2\leq k < 2\pi\sqrt{n} \\ \gcd(k,6)=1}} \frac{\left|\mathcal{K}_k(n)\right|}{k^2} \left|I_1\left(\frac{2\pi\sqrt{n}}{3k}\right)\right| \leq \frac{\pi}{6\sqrt{n}} \sum_{2 \leq k < 2\pi \sqrt{n}} \frac{1}{k}e^{\frac{\pi \sqrt{n}}{3}} \leq \frac{\pi^2}{6} e^{\frac{\pi \sqrt{n}}{3}} < \frac{\pi^2}{12} e^{\frac{\sqrt{3}\pi}{3}\sqrt{n}}.
	\]
	Here we used for the last step the elementary inequality
	\[
	e^{\frac{\pi \sqrt{n}}{3}} < \frac{1}{2}e^{\frac{\sqrt{3}\pi}{3}\sqrt{n}}, \quad (n \geq 1).
	\]
	We set $a_\nu = \frac{\pi}{2}\left(\nu - \frac{1}{6}\right).$ First of all by Corollary \ref{BesselIntegralBound} we have for $k=2$ that
	
	\begin{align*}
	\left|\mathcal{I}_{\frac{1}{18},2,\nu}(n)\right| &\leq 2|\sec(a_\nu)|I_1\left(\frac{\pi}{3}\sqrt{n}\right),\\
	\left|\mathcal{I}_{\frac{5}{36},2,\nu}(n)\right| &\leq 2|\sec(a_\nu)|I_1\left(\frac{\pi\sqrt{10}}{6}\sqrt{n}\right),\\
	\left|\mathcal{I}_{\frac{1}{6},2,\nu}(n)\right| &\leq 2|\sec(a_\nu)|I_1\left(\frac{\sqrt{3}\pi}{3}\sqrt{n}\right).
	\end{align*}
	Since the Bessel function $I_1(x)$ is monotone increasing for positive $x$ we can bound all Bessel functions appearing in the sums uniformly in $b,k,\nu$ by $e^{\frac{\sqrt{3}\pi}{3}\sqrt{n}}$ against $I_1\left(\frac{\sqrt{3}\pi}{3}\sqrt{n}\right)$.
	
	Further, we recall that
	\[
	\max\left(\frac{5\pi}{36\sqrt{6n}},\frac{\pi}{6\sqrt{6n}},\frac{\pi}{6\sqrt{n}},\frac{\pi}{18\sqrt{6n}}\right)=\frac{\pi}{6\sqrt{n}}.
	\]
	
	Combining this with the trivial estimate on the Kloosterman sums and Lemma \ref{BesselFunctionUpperBounds} shows that the remaining three sums from the statement added up are bounded by
	\[
	\frac{\pi}{3\sqrt{n}}e^{\frac{\sqrt{3}\pi}{3}\sqrt{n}}\sum_{2\leq k < 2\pi \sqrt{n}} \frac{1}{k} \sum_{\nu = 1}^{k} \left|\sec\left(\frac{\pi}{k}\left(\nu - \frac{1}{6}\right)\right)\right|
	\]
	We apply Lemma \ref{lemma:bound-cos-sum} and the elementary inequality $\log(k) \leq 3k^{\frac{1}{8}}$ for all $k \geq 1$ to find that the previous expression is bounded from above by
	\[
	\frac{8\pi}{3\sqrt{n}} e^{\frac{\sqrt{3}\pi}{3}\sqrt{n}} \sum_{2\leq k < 2\pi \sqrt{n}} \log(k)\leq \frac{8\pi}{\sqrt{n}} e^{\frac{\sqrt{3}\pi}{3}\sqrt{n}} \sum_{2\leq k < 2\pi \sqrt{n}} k^{\frac{1}{8}} \leq 16\pi^2(2\pi \sqrt{n})^{\frac{1}{8}}.
	\]
	Therefore, all the sums added up are bounded by
	\[
	\left(\frac{\pi^2}{12}+ 16\pi^2(2\pi\sqrt{n})^{\frac{1}{8}}\right)e^{\frac{\sqrt{3}\pi}{3}\sqrt{n}} \leq 200 n^{\frac{1}{16}}e^{\frac{\sqrt{3}\pi}{3}\sqrt{n}}.
	\]
\end{proof}

For two functions $f,g$ we introduce the non-standard notation 

$$f = O_{\leq}(g),$$

\noindent
if $f=O(g)$ and the implicit constant can be chosen to be $1.$ This notation has for instance been used in the proof of log-concavity for the plane partition function $\pl(n)$ in \cite{OnoPujahariRolen}.
Combing the results we obtain the following Lemma.

\begin{lemma}\label{LehmerBound}
	Let $n\geq 1.$ Then,
	
	$$p_2(n)= \frac{\pi}{6\sqrt{n}}I_1\left(\frac{2\pi\sqrt{n}}{3}\right) + \frac{\pi}{18\sqrt{6n}}\mathcal{I}_{\frac{1}{18},1,0}(n) + O_{\leq}\left(200n^{\frac{1}{16}}e^{\frac{\sqrt{3}\pi}{3}\sqrt{n}} + 46500n^{\frac{15}{16}}\right).$$
\end{lemma}

\section{Asymptotic formula for $p_2(n)$}

For proving log-concavity strong asymptotics are required. Thus, we will have to use a strong asymptotic formula for the Bessel function and $\mathcal{I}_{\frac{1}{18},1,0}(n).$ The latter is the more difficult problem. We will find such a formula with the saddle point method. We define the following list of real numbers:

\begin{align*}
&a_1 := \frac{1}{4 \sqrt{3}}, \quad a_2 := \frac{1}{18 \sqrt{2}}, \quad a_3 := -\frac{3 \sqrt{3}}{64 \pi
}, \quad a_4:= -\frac{324+5 \pi ^2}{3888\sqrt{2}\pi}, \quad a_5 := -\frac{45 \sqrt{3} }{2048 \pi^2}  \quad a_6:= \frac{1080+17 \pi ^2}{186624 \sqrt{2}}, \\
& a_7 := -\frac{945 \sqrt{3}}{32768 \pi ^3}, \quad a_8 := -\frac{349920+33048 \pi ^2+455 \pi
^4}{40310784 \sqrt{2} \pi}, \quad a_9:= -\frac{127575 \sqrt{3}}{2097152
\pi ^4}.
\end{align*}
Our goal for this section is to establish the following asymptotic formula for $p_2(n)$.
\begin{theorem}\label{AsymptoticFormula}
	Let $n\geq 1.$ Then, 
	
	\begin{align}
		p_2(n) = \sum_{k=1}^{9} a_k n^{-\frac{k+2}{4}}e^{\frac{2\pi}{3}\sqrt{n}} + O_{\leq}\left(15\frac{e^{\frac{2\pi}{3}\sqrt{n}}}{n^3}\right).
	\end{align}
\end{theorem}
\begin{remark}\label{rmk:asymptotic-formula}
	In \cite{BringmannMahlburg} Bringmann and Mahlburg showed for any $0< c < \frac{1}{8}$ the asymptotic fomula
	\[
		p_2(n) = \left(\frac{1}{4\sqrt{3}n^\frac{3}{4}}+\frac{1}{18\sqrt{2}n}\right)e^{\frac{2\pi}{3}\sqrt{n}} + O\left(\frac{e^{\frac{2\pi}{3}\sqrt{n}}}{n^{1+c}}\right)
	\]
	as $n \rightarrow \infty$.
\end{remark}
To prove this Theorem we will have to find asymptotic expansions with explicit error for the terms
\[
	\frac{\pi}{6\sqrt{n}}I_1\left(\frac{2\pi\sqrt{n}}{3}\right) , \quad \frac{\pi}{18\sqrt{6n}}\mathcal{I}_{\frac{1}{18},1,0}(n).
\]
This is because these term turn out to grow asymptotically for large $n$ like $\sim e^{\frac{2\pi}{3}\sqrt{n}}$. The first step is to understand well the asymptotic behaviour of Bessel functions. Define for non-negative integers $\nu$ and $n$, as well as a complex number $\lambda \in \C\setminus\Z$, the following expressions
\begin{align*}
	(\lambda)_n &\coloneqq \frac{\Gamma(\lambda + n)}{\Gamma(\lambda)},\\
	a_n(\nu) &\coloneqq (-1)^n \frac{\left(\frac{1}{2}-\nu\right)_n\left(\frac{1}{2}+\nu\right)_n}{2^n n!}.
\end{align*}
We then have the divergent asymptotic expansion (see. for instance \cite{Temme}) for positive large $x$
\[
	I_\nu(x) \sim \frac{e^x}{\sqrt{2\pi x}}\sum_{n=0}^{\infty} (-1)^n\frac{a_n(\nu)}{x^n}.
\]
In \cite[Theorem 3.9]{Banerjee} Banerjee has made this result effective, in the sense that he gave explicit bounds on the error of the asymptotic expansion up to any order. To state his result, we define following \cite{Banerjee} for $\nu \in \N_{\geq 0}$ and integral $N \geq 1$ the two functions
\begin{align*}
E_{1,1}^N &= \left(1 + \frac{(2\nu +1)(\nu +2)}{\log(N+1)} + \frac{(2\nu + 1)(\nu + 2)}{N+2}\right),\\
E_{\nu}^N &= \frac{1}{\sqrt{2\pi}}E_{\nu,1}^N + \frac{1}{\log(N+1)}\left(\sqrt{2} + \frac{1}{\sqrt{\nu + N + \frac{3}{2}}}\right).
\end{align*}
In the case that $\nu \leq N$ Banerjee's result \cite[Theorem 3.9]{Banerjee} states that for all $x\geq 1$ the following bound holds
\begin{equation}\label{BanerjeesBound}
	\left|I_{\nu}(x) - \frac{e^x}{\sqrt{2\pi x}} \sum_{k=0}^{N} \frac{(-1)^k a_k(\nu)}{x^k}\right| < \frac{e^x}{\sqrt{2\pi x}}E_{\nu}^N \frac{|a_{N+1}(\nu)|}{x^{N+1}}.
\end{equation}
We introduce for a non-negative integer $n$ the shorthand notation $a(n) \coloneqq a_n(1)$. Applying \eqref{BanerjeesBound} with $N=4$ immediately gives the following expression for the Bessel function 
\begin{multline}\label{eqn:asymptotic-expansion-Bessel-function}
\frac{\pi}{6\sqrt{n}}I_1\left(\frac{2\pi\sqrt{n}}{3}\right) = \left(\frac{1}{4\sqrt{3}n^{\frac{3}{4}}} - \frac{3\sqrt{3}}{64\pi n^{\frac{5}{4}}} - \frac{45\sqrt{3}}{2048\pi^2n^{\frac{7}{4}}} - \frac{945\sqrt{3}}{32768\pi^3n^{\frac{9}{4}}} - \frac{127575\sqrt{3}}{2097152\pi^4n^{\frac{11}{4}}}\right)e^{\frac{2\pi}{3}\sqrt{n}}\\+O_{\leq}\left(\frac{5893965\sqrt{3}\left(\frac{\frac{14}{5}+\frac{9}{\log(4)}}{\sqrt{2\pi}} + \frac{\sqrt{\frac{2}{11}}+\sqrt{2}}{\log(4)}\right)}{33554432\pi^5n^{\frac{13}{4}}}e^{\frac{2\pi}{3}\sqrt{n}}\right).
\end{multline}

The second term however requires more work. We briefly recall that
\[
\mathcal{I}_{\frac{1}{18},1,0}(n) := \int_{-1}^{1} \frac{\sqrt{1-x^2}I_1\left(\frac{2\pi}{3}\sqrt{(1-x^2)n}\right)}{\cosh\left(-\frac{\pi i}{6}-\pi\sqrt{\frac{1}{54}}x\right)} dx
\]
By Lemma \ref{IntegrandSimplification} we can rewrite this as
\[
\mathcal{I}_{\frac{1}{18},1,0}(n) = \int_{0}^{1} 4\frac{\cos\left(\frac{\pi}{6}\right)\cosh\left(\pi\sqrt{\frac{1}{54}}x\right)}{\cos\left(\frac{\pi}{3}\right)+\cosh\left(2\pi\sqrt{\frac{1}{54}}x\right)}\sqrt{1-x^2}I_1\left(\frac{2\pi}{3}\sqrt{(1-x^2)n}\right)dx
\]
By the asymptotics of the Bessel function we see immediately that the main contribution comes from the points close to $0$. It thus makes sense for an $0<R<1$ to split the integral into two integrals at the point $R.$ We plug in the asymptotic expansion of the Bessel function \eqref{BanerjeesBound} and obtain that for any integer $N\geq 0$ and $n \geq \frac{9}{4\pi^2(1-R^2)}$ we can write
\begin{multline*}
	\mathcal{I}_{\frac{1}{18},1,0}(n) = \frac{1}{\sqrt{2\pi}}\sum_{k=0}^{N} \frac{(-1)^ka(k)}{(\frac{2\pi}{3}\sqrt{n})^{k+\frac{1}{2}}} \int_0^R 4\frac{\cos\left(\frac{\pi}{6}\right)\cosh\left(\pi\sqrt{\frac{1}{54}}x\right)}{\cos\left(\frac{\pi}{3}\right)+\cosh\left(2\pi\sqrt{\frac{1}{54}}x\right)}(1-x^2)^{\frac{1}{4}-\frac{k}{2}} e^{\frac{2\pi}{3}\sqrt{n(1-x^2)}} dx \\
	+ O_{\leq}\left(\frac{1}{\sqrt{2\pi}}\frac{E_{1}^N|a(N+1)|}{(\frac{2\pi}{3}\sqrt{n})^{N+\frac{3}{2}}}\int_0^R 4\frac{\cos\left(\frac{\pi}{6}\right)\cosh\left(\pi\sqrt{\frac{1}{54}}x\right)}{\cos\left(\frac{\pi}{3}\right)+\cosh\left(2\pi\sqrt{\frac{1}{54}}x\right)}(1-x^2)^{\frac{1}{4}-\frac{N+1}{2}} e^{\frac{2\pi}{3}\sqrt{n(1-x^2)}} dx\right) 
	\\+ O_{\leq}\left(\int_{R}^{1} 4\frac{\cos\left(\frac{\pi}{6}\right)\cosh\left(\pi\sqrt{\frac{1}{54}}x\right)}{\cos\left(\frac{\pi}{3}\right)+\cosh\left(2\pi\sqrt{\frac{1}{54}}x\right)}\sqrt{1-x^2}I_1\left(\frac{2\pi}{3}\sqrt{(1-x^2)n}\right)dx\right)
\end{multline*}

We will now transform these integral into a for our purposes simpler shape following a change of variables suggested by Caner Nazaroglu. To this end we define for $0\leq x <1$ and integers $k\geq 1$ the function
\[
	h(x,k)\coloneqq 4\frac{\cos\left(\frac{\pi}{6}\right)\cosh\left(\pi\sqrt{\frac{1}{54}}\sqrt{1-(1-x^2)^2}\right)}{\cos\left(\frac{\pi}{3}\right)+\cosh\left(2\pi\sqrt{\frac{1}{54}}\sqrt{1-(1-x^2)^2}\right)}\frac{(1-x^2)^{\frac{3}{2}-k}}{\sqrt{2-x^2}}.
\]
\begin{lemma}\label{IntegralTransformation}
	Let $0<R<1$ and let $k$ be a positive integer. We then have
	\[
		\int_{0}^{R} 4\frac{\cos\left(\frac{\pi}{6}\right)\cosh\left(\pi\sqrt{\frac{1}{54}}x\right)}{\cos\left(\frac{\pi}{3}\right)+\cosh\left(2\pi\sqrt{\frac{1}{54}}x\right)}(1-x^2)^{\frac{1}{4}-\frac{k}{2}} e^{2\pi\sqrt{n(1-x^2)}} dx
		= 2e^{\frac{2\pi}{3}\sqrt{n}} \int_{0}^{\sqrt{1-\sqrt{1-R^2}}} h(x,k) e^{-\frac{2\pi}{3}\sqrt{n}x^2} dx.
	\]
\end{lemma}

\begin{proof}
	We do the change of variables $u=\sqrt{1-x^2}-1$. Then, $dx=-\frac{u+1}{\sqrt{1-(u+1)^2}}du$ and $x=\sqrt{1-(u+1)^2}$. It follows that
	\begin{multline*}
	\int_{0}^{R} 4\frac{\cos\left(\frac{\pi}{6}\right)\cosh\left(\pi\sqrt{\frac{1}{54}}x\right)}{\cos\left(\frac{\pi}{3}\right)+\cosh\left(2\pi\sqrt{\frac{1}{54}}x\right)}(1-x^2)^{\frac{1}{4}-\frac{k}{2}} e^{2\pi\sqrt{n(1-x^2)}} dx\\
	= -e^{\frac{2\pi}{3}\sqrt{n}} \int_{0}^{\sqrt{1-R^2}-1} 4\frac{\cos\left(\frac{\pi}{6}\right)\cosh\left(\pi\sqrt{\frac{1}{54}}\sqrt{1-(u+1)^2}\right)}{\cos\left(\frac{\pi}{3}\right)+\cosh\left(2\pi\sqrt{\frac{1}{54}}\sqrt{1-(u+1))^2}\right)}\frac{(u+1)^{\frac{3}{2}-k}}{\sqrt{1-(u+1)^2}} e^{\frac{2\pi}{3}\sqrt{n}u} du
	\end{multline*}
	Finally, we perform the change of variables $u=-z^2$. Then, $du=-2zdz$ and the last integral becomes
	\[
		2e^{\frac{2\pi}{3}\sqrt{n}} \int_{0}^{\sqrt{1-\sqrt{1-R^2}}} 4\frac{\cos\left(\frac{\pi}{6}\right)\cosh\left(\pi\sqrt{\frac{1}{54}}\sqrt{1-(1-z^2)^2}\right)}{\cos\left(\frac{\pi}{3}\right)+\cosh\left(2\pi\sqrt{\frac{1}{54}}\sqrt{1-(1-z^2)^2}\right)}\frac{(1-z^2)^{\frac{3}{2}-k}z}{\sqrt{1-(1-z^2)^2}} e^{-\frac{2\pi}{3}\sqrt{n}z^2} dz
	\]
	We finally rewrite $\sqrt{1-(1-z^2)^2}=z\sqrt{2-z^2}$ and this finishes the proof.
\end{proof}

We set $r\coloneqq \sqrt{1-\sqrt{1-R^2}}$. The advantage of this integral representation is that we already have a standard Gaussian integral. Therefore, we can write with Lemma \ref{IntegralTransformation}
\begin{multline}\label{BesselIntegralExpanded}
	\mathcal{I}_{\frac{1}{18},1,0}(n) = \sqrt{\frac{2}{\pi}}e^{\frac{2\pi}{3}\sqrt{n}}\sum_{k=0}^{N} \frac{(-1)^ka(k)}{(\frac{2\pi}{3}\sqrt{n})^{k+\frac{1}{2}}} \int_0^r h(x,k) e^{-\frac{2\pi}{3}\sqrt{n}x^2} dx \\
	+ O_{\leq}\left(\sqrt{\frac{2}{\pi}}\frac{E_{1}^N|a(N+1)|}{(\frac{2\pi}{3}\sqrt{n})^{N+\frac{3}{2}}}e^{\frac{2\pi}{3}\sqrt{n}}\int_0^r h(x,N+1) e^{-\frac{2\pi}{3}\sqrt{n}x^2} dx\right) 
	\\+ O_{\leq}\left(\int_{R}^{1} 4\frac{\cos\left(\frac{\pi}{6}\right)\cosh\left(\pi\sqrt{\frac{1}{54}}x\right)}{\cos\left(\frac{\pi}{3}\right)+\cosh\left(2\pi\sqrt{\frac{1}{54}}x\right)}\sqrt{1-x^2}I_1\left(\frac{2\pi}{3}\sqrt{(1-x^2)n}\right)dx\right), \text{where } n\geq \frac{9}{4\pi^2(1-R^2)}.
\end{multline}

We will now show that $h(z,k)$ is well-defined and analytic in the disk $|z|\leq r$, where we extended $h(z,k)$ in the natural way for complex inputs. The first observation is that the term $\sqrt{1-(1-z^2)^2}$ is itself not continous in the disk due to the branch cut of $\sqrt{z}$ along the negative real axis. However, since $\cosh(z)$ is an even function this problem disappears. The slightly more difficult part is to show that for all $0\leq r <1$ the following expression is bounded from below
\[
\inf_{|z|=r} \left|\cos\left(\frac{\pi}{3}\right)+\cosh\left(2\pi\sqrt{\frac{1}{54}}\sqrt{1-(1-z^2)^2}\right)\right| > 0 
\]
This is the content of the following lemma.

\begin{lemma}\label{LowerBoundCoshTerm}
	Let $0\leq r \leq 1$. We have
	\begin{align*}
		&\inf_{|z|=r} \left|\cos\left(\frac{\pi}{3}\right)+\cosh\left(2\pi\sqrt{\frac{1}{54}}\sqrt{1-(1-z^2)^2}\right)\right|\\ &= \cos\left(\frac{\pi}{3}\right) + \cos\left(\frac{2\pi}{\sqrt{54}}\sqrt{\frac{1}{2}\left(\sqrt{4r^4+4r^6+r^8} + 2r^2+r^4\right)}\right).
	\end{align*}
\end{lemma}

\begin{proof}
	 We estimate
	\[
		\left|\cos\left(\frac{\pi}{3}\right)+\cosh\left(2\pi\sqrt{\frac{1}{54}}\sqrt{1-(1-z^2)^2}\right)\right| \geq \left|\re{\cos\left(\frac{\pi}{3}\right)+\cosh\left(2\pi\sqrt{\frac{1}{54}}\sqrt{1-(1-z^2)^2}\right)}\right|
	\]
	Recall that $\cosh(x+iy) = \cosh(x)\cos(y)+i\sinh(x)\sin(y)$ and thus $\re{\cosh(x+iy)} = \cosh(x)\cos(y)$. Since $\cosh(x) \geq 1$ for all real numbers $x$, it follows that
	\[
	\left|\re{\cosh\left(2\pi\sqrt{\frac{1}{54}}\sqrt{1-(1-z^2)^2}\right)}\right| \geq \left|\cos\left(2\pi\sqrt{\frac{1}{54}}\im{\sqrt{1-(1-z^2)^2}}\right)\right|. 
	\]
	We write $z=re^{i\theta}$ with $0\leq \theta < 2\pi$. The last quantity becomes after expanding the binomial term in the square root
	\begin{equation}\label{eqn:cos}
		= \left|\cos\left(2\pi\sqrt{\frac{1}{54}}\im{\sqrt{2r^2e^{2i\theta} - r^4e^{4i\theta}}\right)}\right|.
	\end{equation}
	Recall that for the principal branch of the square root we have the following well-known algebraic expression in cartesian coordinates
	\[
		\sqrt{x+iy} = \sqrt{\frac{1}{2}\left(\sqrt{x^2+y^2}+x\right)} + i\sigma(y)\sqrt{\frac{1}{2}\left(\sqrt{x^2+y^2}-x\right)}, \quad \sigma(y) \begin{cases}
		1 & \text{if } y \geq 0,\\
		-1 & \text{if } y < 0.
		\end{cases}
	\]
	We briefly remark that the sign of the imaginary part is irrelevant for us since $\cos(z)$ is an even function. This enables us to ignore the monodromy of the square root and without loss of generality we will assume that $\sigma=1$. Hence, using Euler's formula we write
	\begin{multline*}
		\im{\sqrt{2r^2e^{2i\theta} - r^4e^{4i\theta}}} 
		\\= \sqrt{\frac{1}{2}\left(\sqrt{\left(2r^2\cos(2\theta) - r^4\cos(4\theta)\right)^2+\left(2r^2\sin(2\theta) - r^4\sin(4\theta)\right)^2}-\left(2r^2\cos(2\theta)-r^4\cos(4\theta)\right)\right)}.
	\end{multline*}
	By monotonicity of the real square root, we will maximize both summands separately. For the first term we have
	\begin{multline*}
	\left(2r^2\cos(2\theta) - r^4\cos(4\theta)\right)^2+\left(2r^2\sin(2\theta) - r^4\sin(4\theta)\right)^2 \\= 4r^4\left(\cos(2\theta)^2+\sin(2\theta)^2\right) + r^8\left(\cos(4\theta)^2+\sin(4\theta)^2\right) - 4r^6\left(\cos(2\theta)\cos(4\theta)+\sin(2\theta)\sin(4\theta)\right)\\
	=4r^4+r^8-4r^6\cos(2\theta).
	\end{multline*}
	Thus, the first term has a maximum at $\theta\in\{\frac{\pi}{2},\frac{3\pi}{2}\}$. The same is true for the second term, namely its maximum is $2r^2+r^4$. This proves that 
	\[
		\left|\im{\sqrt{2r^2e^{2i\theta} - r^4e^{4i\theta}}}\right| \leq \sqrt{\frac{1}{2}\left(\sqrt{4r^4+4r^6+r^8} + 2r^2+r^4\right)}.
	\]
	and the right hand side is monotonically increasing in $r$. At $r=1$ the right hand side is $\sqrt{3}$ and since $\frac{2\pi\sqrt{3}}{\sqrt{54}}<\frac{\pi}{2}$, we can a fortiori remove the absolute values in \eqref{eqn:cos} as $\cos(x)$ is monotonically decreasing on $[0,\frac{\pi}{2}]$. Furthermore, we have proven  the lower bound
	\begin{align*}
	&\inf_{|z|=r} \left|\cos\left(\frac{\pi}{3}\right)+\cosh\left(2\pi\sqrt{\frac{1}{54}}\sqrt{1-(1-z^2)^2}\right)\right|\\ &\geq \cos\left(\frac{\pi}{3}\right) + \cos\left(\frac{2\pi}{\sqrt{54}}\sqrt{\frac{1}{2}\left(\sqrt{4r^4+4r^6+r^8} + 2r^2+r^4\right)}\right).
	\end{align*}
	Finally, we notice that in the case $r>0$ the inequality becomes sharp for $\theta = \frac{\pi}{2}$ and $\theta=\frac{3\pi}{2}$.
\end{proof}

\begin{prop}\label{TaylorExpansionBound}
	The function $h(z,k)$ is analytic in the disk $|z| < 1$ and for any $k,N \geq 1$ we have in the smaller disk $|z| \leq \frac{1}{4}$ the approximation
	\[
		\left|h(z,k) - \sum_{n=0}^{N-1} \frac{h^{(n)}(0,k)}{n!}z^n\right| \leq 2\left(\frac{4}{3}\right)^{N-1}\frac{4\cos\left(\frac{\pi}{6}\right) e^{\sqrt{\frac{3}{2}}\pi\frac{41}{256}}}{\frac{1}{2} + \cos\left(\frac{1}{8}\sqrt{\frac{41}{6}}\pi\right)}M(k)z^{N},
	\]
	where for $k=0,1$ we set  $M(k)=\left(\frac{25}{16}\right)^{\frac{3}{2}-k}$ and for $k\geq 2$ we set
	$M(k)=\left(\frac{7}{16}\right)^{\frac{3}{2}-k}$.
\end{prop}

\begin{proof}
	The fact that $h(z,k)$ is analytic in the disk $|z| < 1$ follows immediately from the preceeding Lemma and the fact that $\cosh(z)$ is an even function. Since $h(z,k)$ is in particular holomorphic in the disk $|z| \leq \frac{3}{4}$ a simple application of the residue theorem shows the following equality (see for instance \cite{BringmannJenningsMahlburg})
	\[
		h(z,k) = \sum_{\ell=0}^{N-1} \frac{h^{(\ell)}(0,k)}{\ell!}z^\ell + \frac{z^N}{2\pi i}\int_{|z|=\frac{3}{4}} \frac{h(w,k)}{w^N(w-z)} dw,
	\]
	for all $z$ with $|z|< \frac{3}{4}$, where the circle is traversed in the counterclockwise direction. We can then estimate
	\[
		\left|\sum_{\ell=N}^{\infty} \frac{h^{(\ell)}(z,k)}{\ell!}z^{\ell}\right| \leq \left(\frac{4}{3}\right)^{N-1} \sup_{|w|=\frac{3}{4}} \left|\frac{h(w,k)}{w-z}\right|.
	\]
	Suppose from now on that $|z| \leq \frac{1}{4}$. It follows immediately that
	\[
		\sup_{|w|=\frac{3}{4}} \left|\frac{1}{w-z}\right| \leq 2.
	\]
	Hence, it suffices to show the upper bound
	\[
		\sup_{|w|=\frac{3}{4}} |h(w,k)| \leq \frac{4\cos\left(\frac{\pi}{6}\right) e^{\sqrt{\frac{3}{2}}\pi\frac{41}{256}}}{\frac{1}{2} + \cos\left(\frac{1}{8}\sqrt{\frac{41}{6}}\pi\right)}M(k).
	\]
	We will accomplish this by maximizing all factors individually. First of all we write $w=re^{i\theta}$ with $r = \frac{3}{4}$ and $0\leq \theta < 2\pi$. Then,
	\[
		4\cos\left(\frac{\pi}{6}\right)\left|\cosh\left(\frac{\pi}{\sqrt{54}}\sqrt{1-(1-r^2e^{2i\theta})}\right)\right| \leq 4\cos\left(\frac{\pi}{6}\right)e^{\frac{\pi}{\sqrt{54}}\sqrt{\left|1-(1-r^2e^{2i\theta})^2\right|}}.
	\]
	We estimate further
	\[
		\left|1-(1-r^2e^{2i\theta})^2\right| = r^2\left|2-r^2e^{2i\theta}\right| \leq r^2(r^2+2)=r^4+2r^2.
	\]
	Thus, we conclude the upper bound
	\[
		\sup_{|w|=\frac{3}{4}} 4\cos\left(\frac{\pi}{6}\right)\left|\cosh\left(\frac{\pi}{\sqrt{54}}\sqrt{1-(1-z^2)^2}\right)\right| \leq 4\cos\left(\frac{\pi}{6}\right) e^{\frac{\pi}{\sqrt{54}}(r^4+2r^2)} = 4\cos\left(\frac{\pi}{6}\right) e^{\sqrt{\frac{3}{2}}\pi\frac{41}{256}}.
	\]
	where for the inequality we used as well that $x \leq x^2$ for positive $x \geq 1$ to cancel the square root.
	By Lemma \ref{LowerBoundCoshTerm} we have that
	\[\inf_{|w|=\frac{3}{4}} \left|\cos\left(\frac{\pi}{3}\right)+\cosh\left(2\pi\sqrt{\frac{1}{54}}\sqrt{1-(1-w^2)^2}\right)\right| = \frac{1}{2} + \cos\left(\frac{1}{8}\sqrt{\frac{41}{6}}\pi\right).
	\]
	Finally, we have that
	\[
	\sup_{|w|=\frac{3}{4}} \left|(1-w^2)^{\frac{3}{2}-k}\right| \leq M(k), \sup_{|w|=\frac{3}{4}} \left|(2-w^2)^{-\frac{1}{2}}\right| \leq 1.
	\]
	This concludes the proof.
\end{proof}

We are finally ready to prove Theorem \ref{AsymptoticFormula}.

\begin{proof}[Proof of Theorem \ref{AsymptoticFormula}]
We recall that by Lemma \ref{LehmerBound} we have for all $n\geq 1$ the equality
\[p_2(n)= \frac{\pi}{6\sqrt{n}}I_1\left(\frac{2\pi\sqrt{n}}{3}\right) + \frac{\pi}{18\sqrt{6n}}\mathcal{I}_{\frac{1}{18},1,0}(n) + O_{\leq}\left(200n^{\frac{1}{16}}e^{\frac{\sqrt{3}\pi}{3}\sqrt{n}} + 46500n^{\frac{15}{16}}\right).
\]
By virtue of \eqref{eqn:asymptotic-expansion-Bessel-function} we have
\begin{multline*}
\frac{\pi}{6\sqrt{n}}I_1\left(\frac{2\pi\sqrt{n}}{3}\right) = \left(\frac{1}{4\sqrt{3}n^{\frac{3}{4}}} - \frac{3\sqrt{3}}{64\pi n^{\frac{5}{4}}} - \frac{45\sqrt{3}}{2048\pi^2n^{\frac{7}{4}}} - \frac{945\sqrt{3}}{32768\pi^3n^{\frac{9}{4}}} - \frac{127575\sqrt{3}}{2097152\pi^4n^{\frac{11}{4}}}\right)e^{\frac{2\pi}{3}\sqrt{n}}\\+O_{\leq}\left(\frac{5893965\sqrt{3}\left(\frac{\frac{14}{5}+\frac{9}{\log(4)}}{\sqrt{2\pi}} + \frac{\sqrt{\frac{2}{11}}+\sqrt{2}}{\log(4)}\right)}{33554432\pi^5n^{\frac{13}{4}}}e^{\frac{2\pi}{3}\sqrt{n}}\right).
\end{multline*}
We will now show a sufficiently strong approximation for the term $\frac{\pi}{18\sqrt{6n}}\mathcal{I}_{\frac{1}{18},1,0}(n)$. We use \eqref{BesselIntegralExpanded} with $N=3,$ as well as $R= \frac{\sqrt{31}}{16}$ so that $r=\sqrt{1-\sqrt{1-R^2}} = \frac{1}{4}$ and obtain that for all $n \geq 1$ we have
\begin{multline*}
\frac{\pi}{18\sqrt{6n}}\mathcal{I}_{\frac{1}{18},1,0}(n) = \frac{\pi}{18\sqrt{6n}}\sqrt{\frac{2}{\pi}}e^{\frac{2\pi}{3}\sqrt{n}}\sum_{k=0}^{3} \frac{(-1)^ka(k)}{(\frac{2\pi}{3}\sqrt{n})^{k+\frac{1}{2}}} \int_{0}^{\frac{1}{4}} h(x,k) e^{-\frac{2\pi}{3}\sqrt{n}x^2} dx \\
+O_{\leq}\left(\frac{\pi}{18\sqrt{6n}}\sqrt{\frac{2}{\pi}}\frac{E_{1}^3|a(4)|}{(\frac{2\pi}{3}\sqrt{n})^{\frac{11}{2}}}e^{\frac{2\pi}{3}\sqrt{n}}\int_{0}^{\frac{1}{4}} h(x,4) e^{-\frac{2\pi}{3}\sqrt{n}x^2} dx\right) 
\\+ O_{\leq}\left(\frac{\pi}{18\sqrt{6n}}\int_{\frac{\sqrt{31}}{16}}^{1} 4\frac{\cos\left(\frac{\pi}{6}\right)\cosh\left(\pi\sqrt{\frac{1}{54}}x\right)}{\cos\left(\frac{\pi}{3}\right)+\cosh\left(2\pi\sqrt{\frac{1}{54}}x\right)}\sqrt{1-x^2}I_1\left(\frac{2\pi}{3}\sqrt{(1-x^2)n}\right)dx\right).
\end{multline*}
We consider the sum first. We now expand each term into a power series and estimate the error of each term with the help of Proposition \ref{TaylorExpansionBound}.
\begin{multline*}
\frac{\pi}{18\sqrt{6n}}\mathcal{I}_{\frac{1}{18},1,0}(n) = e^{\frac{2\pi}{3}\sqrt{n}}\sum_{k=0}^{3} \left(\sum_{\ell=0}^{7-2k} \frac{b_{k,\ell}}{n^{\frac{k}{2}+\frac{3}{4}}}\int_{0}^{\frac{1}{4}}x^l e^{-\frac{2\pi}{3}\sqrt{n}x^2}dx + O_{\leq}\left(\frac{|c_k|}{n^{\frac{k}{2}+\frac{3}{4}}} \int_{0}^{\frac{1}{4}} x^{8-2k}e^{-\frac{2\pi}{3}\sqrt{n}x^2}\right)\right)\\
+O_{\leq}\left(\frac{\pi}{18\sqrt{6n}}\sqrt{\frac{2}{\pi}}\frac{E_{1}^3|a(4)|}{(\frac{2\pi}{3}\sqrt{n})^{\frac{11}{2}}}e^{\frac{2\pi}{3}\sqrt{n}}\int_{0}^{\frac{1}{4}} h(x,4) e^{-\frac{2\pi}{3}\sqrt{n}x^2} dx\right) 
\\+ O_{\leq}\left(\frac{\pi}{18\sqrt{6n}}\int_{\frac{\sqrt{31}}{16}}^{1} 4\frac{\cos\left(\frac{\pi}{6}\right)\cosh\left(\pi\sqrt{\frac{1}{54}}x\right)}{\cos\left(\frac{\pi}{3}\right)+\cosh\left(2\pi\sqrt{\frac{1}{54}}x\right)}\sqrt{1-x^2}I_1\left(\frac{2\pi}{3}\sqrt{(1-x^2)n}\right)dx\right),
\end{multline*}
where we defined 
\[
b_{k,\ell} \coloneqq \frac{\pi}{18\sqrt{6}}\sqrt{\frac{2}{\pi}}\frac{(-1)^ka(k)}{(\frac{2\pi}{3})^{k+\frac{1}{2}}}\frac{h^{(\ell)}(0,k)}{\ell!}, \hspace{0.5em} c_k \coloneqq 2\left(\frac{4}{3}\right)^{7-2k}\frac{4\cos\left(\frac{\pi}{6}\right) e^{\sqrt{\frac{3}{2}}\pi\frac{41}{256}}}{\frac{1}{2} + \cos\left(\frac{1}{8}\sqrt{\frac{41}{6}}\pi\right)}M(k)\frac{\pi}{18\sqrt{6n}}\sqrt{\frac{2}{\pi}}\frac{(-1)^ka(k)}{(\frac{2\pi}{3})^{k+\frac{1}{2}}}.
\]
Note that since $h(z,k)$ is an even function for all $k\geq 0$ we can immediately see that $b_k=0$ for $k$ odd. We now extend the range of integration in the Gaussian integrals to the positive real numbers and take into account an extra error term in each summand. To this end we introduce for $k=0,1,2,3$ the notation

\begin{align*}
	G(k)&\coloneqq \sum_{\ell=0}^{7-2k} \frac{b_{k,\ell}}{n^{\frac{k}{2}+\frac{3}{4}}} \int_{0}^{\infty} x^le^{-\frac{2\pi}{3}\sqrt{n}x^2}dx,\\
	g_1(k,n) &\coloneqq \sum_{\ell=0}^{7-2k} \frac{|b_{k,\ell}|}{n^{\frac{k}{2}+\frac{3}{4}}} \int_{\frac{1}{4}}^{\infty}x^l e^{-\frac{2\pi}{3}\sqrt{n}x^2}dx,\\
	 g_2(k,n) &\coloneqq \frac{|c_k|}{n^{\frac{k}{2}+\frac{3}{4}}} \int_{0}^{\infty} x^{8-2k}e^{-\frac{2\pi}{3}\sqrt{n}x^2}dx,\\
	 g_3(n) &\coloneqq \frac{\pi}{18\sqrt{6n}}\sqrt{\frac{2}{\pi}}\frac{E_{1}^3|a(4)|}{(\frac{2\pi}{3}\sqrt{n})^{\frac{11}{2}}}e^{\frac{2\pi}{3}\sqrt{n}}\int_{0}^{\frac{1}{4}} h(x,4) e^{-\frac{2\pi}{3}\sqrt{n}x^2} dx,\\
	 g_4(n) &\coloneqq \frac{\pi}{18\sqrt{6n}}\int_{\frac{\sqrt{31}}{16}}^{1} 4\frac{\cos\left(\frac{\pi}{6}\right)\cosh\left(\pi\sqrt{\frac{1}{54}}x\right)}{\cos\left(\frac{\pi}{3}\right)+\cosh\left(2\pi\sqrt{\frac{1}{54}}x\right)}\sqrt{1-x^2}I_1\left(\frac{2\pi}{3}\sqrt{(1-x^2)n}\right)dx.
\end{align*}
Therefore,
\begin{multline}\label{eqn:abstract-asymptotic-formula-integral}
\frac{\pi}{18\sqrt{6n}}\mathcal{I}_{\frac{1}{18},1,0}(n) = e^{\frac{2\pi}{3}\sqrt{n}}\sum_{k=0}^{3} \left(G(k) + O_{\leq}\left(g_1(k,n)\right)+O_{\leq}\left(g_2(k,n)\right)\right)+O_{\leq}\left(g_3(n)\right) + O_{\leq}\left(g_4(n)\right).
\end{multline}
We obtain the following values
\begin{center}
	\renewcommand{\arraystretch}{1.75}
	\begin{tabular}{| c | c |} 
		\hline
		$k$ & $G(k)$  \\
		\hline
		0 &  $\frac{1}{18 \sqrt{2}n} -\frac{5 \left(81+2 \pi ^2\right)}{7776 \sqrt{2} \pi  n^{3/2}}+\frac{6561+3780 \pi ^2+68 \pi ^4}{746496 \sqrt{2} \pi ^2 n^2}-\frac{5 \left(-1240029+503010 \pi ^2+49572 \pi ^4+728\pi^6\right)}{322486272 \sqrt{2} \pi ^3 n^{5/2}}$ \\
		\hline
		1 & $-\frac{1}{32 \sqrt{2} \pi  n^{3/2}}+\frac{81+10 \pi ^2}{13824
			\sqrt{2} \pi ^2 n^2}-\frac{-10935+1620 \pi ^2+68 \pi ^4}{1327104 \sqrt{2} \pi ^3 n^{5/2}}$  \\
		\hline
		2 & $-\frac{15}{1024 \sqrt{2} \pi ^2 n^2} + \frac{5 \left(10 \pi ^2-243\right)}{147456 \sqrt{2} \pi ^3 n^{5/2}}$  \\
		\hline
		3 & $-\frac{315}{16384 \sqrt{2} \pi ^3 n^{5/2}}$ \\
		\hline
	\end{tabular}
	\renewcommand{\arraystretch}{1}
\end{center}

We are now going to evaluate $g_1(k,n)$ for $k=0,1,2,3.$ First we use the upper bound
\[
|g_1(k,n)| \leq \sum_{\ell=0}^{7-2k} \frac{|b_{k,\ell}|}{n^{\frac{k}{2}+\frac{3}{4}}} \left(e^{-\frac{\sqrt{n}\pi}{24}}\int_{\frac{1}{4}}^{1}x^l dx + \int_{1}^{\infty}x^l e^{-\frac{2\pi}{3}\sqrt{n}x}\right)dx.
\]
By evaluating the integrals in the defining equation we can write
\begin{align}\label{eqn:formulas-g1}
	|g_1(0,n)| &\leq e^{-\frac{2\pi}{3}\sqrt{n}}\sum_{\ell = 0}^{6} \frac{\lambda_{0,\ell}}{n^{\frac{5+2\ell}{4}}} + \frac{e^{-\frac{\sqrt{n}\pi}{24}}}{n^{\frac{3}{4}}}\mu_0, \quad |g_1(1,n)| \leq  e^{-\frac{2\pi}{3}\sqrt{n}}\sum_{\ell = 0}^{4} \frac{\lambda_{1,\ell}}{n^{\frac{7+2\ell}{4}}} + \frac{e^{-\frac{\sqrt{n}\pi}{24}}}{n^{\frac{5}{4}}}\mu_1, \nonumber\\
	|g_1(2,n)| &\leq e^{-\frac{2\pi}{3}\sqrt{n}}\sum_{\ell = 0}^{2} \frac{\lambda_{2,\ell}}{n^{\frac{9+2\ell}{4}}} + \frac{e^{-\frac{\sqrt{n}\pi}{24}}}{n^{\frac{7}{4}}}\mu_2, \quad |g_1(3,n)| \leq  e^{-\frac{2\pi}{3}\sqrt{n}} \frac{\lambda_{3,\ell}}{n^{\frac{11}{4}}} + \frac{e^{-\frac{\sqrt{n}\pi}{24}}}{n^{\frac{9}{4}}}\mu_3,
\end{align}
where the appearing constants $\lambda_{k,\ell}$ and $\mu_k$ are defined by the following tables.
\begin{center}
	\renewcommand{\arraystretch}{1.75}
	\begin{tabular}{| c | c |} 
		\hline
		$k$ & $\mu_k$  \\
		\hline
		$0$ & $\frac{5075695383003+104403159090 \pi ^2+2195009604 \pi ^4+19878040 \pi ^6}{39007939461120 \sqrt{3}}$  \\
		\hline
		$1$ & $\frac{95648445+1762020 \pi ^2+23188 \pi ^4}{1911029760 \sqrt{3} \pi }$  \\
		\hline
		$2$ & $\frac{25 \left(1377 -14 \pi ^2\right)}{1179648 \sqrt{3} \pi ^2}$  \\
		\hline
		$3$ & $\frac{315 \sqrt{3}}{32768 \pi ^3}$ \\
		\hline
	\end{tabular}
	\renewcommand{\arraystretch}{1}
	\end{center}
	\vspace{1cm}
	\begin{center}
	\renewcommand{\arraystretch}{1.75}
	\begin{tabular}{| c | c |} 
		\hline
		$\ell$ & $\lambda_{0,\ell}$  \\
		\hline
		$0$ & $\frac{51904071+2427570 \pi ^2+71604 \pi ^4+728 \pi ^6}{136048896 \sqrt{3} \pi}$  \\
		\hline
		$1$ & $\frac{9624987+1552770 \pi ^2+64260 \pi ^4+728 \pi ^6}{15116544 \sqrt{3} \pi ^2}$  \\
		\hline
		$2$ & $\frac{7499223+5197770 \pi ^2+291924 \pi ^4+3640 \pi ^6}{10077696 \sqrt{3} \pi ^3}$  \\
		\hline
		$3$ & $\frac{-4074381+3739770 \pi ^2+269892 \pi ^4+3640 \pi ^6}{1679616 \sqrt{3} \pi ^4}$ \\
		\hline
		$4$ & $\frac{-5491557+2923290 \pi ^2+255204 \pi ^4+3640 \pi ^6}{373248 \sqrt{3} \pi ^5}$ \\
		\hline
		$5$ & $\frac{5 \left(-1240029+503010 \pi ^2+49572 \pi ^4+728 \pi ^6\right)}{124416 \sqrt{3} \pi ^6}$ \\
		\hline
		$6$ & $\frac{5 \left(-1240029+503010 \pi ^2+49572 \pi ^4+728 \pi ^6\right)}{82944 \sqrt{3} \pi ^7}$ \\
		\hline
	\end{tabular}
	\renewcommand{\arraystretch}{1}
\end{center}
\vspace{1cm}
\begin{center}
	\renewcommand{\arraystretch}{1.75}
	\begin{tabular}{| c | c |} 
		\hline
		$\ell$ & $\lambda_{1,\ell}$  \\
		\hline
		$0$ & $\frac{76545+3780 \pi ^2+68 \pi ^4}{746496 \sqrt{3} \pi ^2}$  \\
		\hline
		$1$ & $\frac{-2187+2700 \pi ^2+68 \pi ^4}{124416 \sqrt{3} \pi ^3}$  \\
		\hline
		$2$ & $\frac{-8019+1980 \pi ^2+68 \pi ^4}{27648 \sqrt{3} \pi ^4}$  \\
		\hline
		$3$ & $\frac{-10935+1620 \pi ^2+68 \pi ^4}{9216 \sqrt{3} \pi ^5}$ \\
		\hline
		$4$ & $\frac{-10935+1620 \pi ^2+68 \pi ^4}{6144 \sqrt{3} \pi ^6}$ \\
		\hline
	\end{tabular}
	\renewcommand{\arraystretch}{1}
\quad
	\renewcommand{\arraystretch}{1.75}
	\begin{tabular}{| c | c |} 
		\hline
		$\ell$ & $\lambda_{2,\ell}$  \\
		\hline
		$0$ & $\frac{5 \left(567 - 10 \pi ^2\right)}{36864 \sqrt{3} \pi ^3}$  \\
		\hline
		$1$ & $\frac{5 \left(243 - 10 \pi ^2\right)}{12288 \sqrt{3} \pi ^4}$  \\
		\hline
		$2$ & $\frac{5 \left(243 - 10 \pi ^2\right)}{8192 \sqrt{3} \pi ^5}$  \\
		\hline
	\end{tabular}
	\renewcommand{\arraystretch}{1}
\quad
	\renewcommand{\arraystretch}{1.75}
	\begin{tabular}{| c | c |} 
		\hline
		$\ell$ & $\lambda_{3,\ell}$  \\
		\hline
		$0$ & $\frac{315 \sqrt{3}}{16384 \pi ^4}$  \\
		\hline
	\end{tabular}
	\renewcommand{\arraystretch}{1}
\end{center}
The term which grows the fastest in \eqref{eqn:formulas-g1} with respect to $n$ is $e^{-\frac{\sqrt{n}\pi}{24}}n^{-\frac{3}{4}}$. A standard calculus argument shows that for $n \geq 31745$ the inequality
\[
e^{-\frac{\sqrt{n}\pi}{24}} \leq n^{-\frac{9}{4}}
\]
holds. Thus, for $n \geq 31745$ we can estimate
\begin{align}\label{eqn:bound-g1}
	\sum_{k=0}^{3} g_1(k,n) \leq \frac{\sum_{k=0}^{3}\left(\mu_k +\sum_{\ell=0}^{6-2k} \lambda_{k,\ell}\right)}{n^3} \leq \frac{7}{5}\frac{1}{n^3}.
\end{align}
We continue to estimate the errors coming from $g_2(k,n).$ These can be explicitly evaluated since they are Gaussian integrals. We have the following table.
\begin{center}
	\renewcommand{\arraystretch}{1.75}
	\begin{tabular}{| c | c |} 
		\hline
		$k$ & $g_2(k,n)$  \\
		\hline
		$0$ & $\frac{4375 e^{\frac{41}{256} \sqrt{\frac{3}{2}} \pi }}{27 \pi ^4 \left(1+2 \cos \left(\frac{1}{8} \sqrt{\frac{41}{6}} \pi
			\right)\right)}\frac{1}{n^3}$  \\
		\hline
		$1$ & $\frac{25 e^{\frac{41}{256} \sqrt{\frac{3}{2}} \pi }}{4 \pi ^4 \left(1+2 \cos \left(\frac{1}{8} \sqrt{\frac{41}{6}} \pi
			\right)\right)}\frac{1}{n^3}$  \\
		\hline
		$2$ & $\frac{45 e^{\frac{41}{256} \sqrt{\frac{3}{2}} \pi }}{32 \sqrt{7} \pi ^4 \left(1+2 \cos \left(\frac{1}{8} \sqrt{\frac{41}{6}} \pi
			\right)\right)}\frac{1}{n^3}$  \\
		\hline
		$3$ & $\frac{135 e^{\frac{41}{256} \sqrt{\frac{3}{2}} \pi }}{128 \sqrt{7} \pi ^4 \left(1+2 \cos \left(\frac{1}{8} \sqrt{\frac{41}{6}}
			\pi \right)\right)}\frac{1}{n^3}$ \\
		\hline
	\end{tabular}
	\renewcommand{\arraystretch}{1}
\end{center}
Therefore, we can estimate
\begin{align}\label{eqn:bound-g2}
\sum_{k=0}^{3} g_2(k,n) = \frac{5 \left(116320+243 \sqrt{7}\right) e^{\frac{41}{256} \sqrt{\frac{3}{2}} \pi }}{3456 \pi ^4\left(1+2 \cos \left(\frac{1}{8} \sqrt{\frac{41}{6}} \pi \right)\right)} \frac{1}{n^3} \leq \frac{8}{5}\frac{1}{n^3}.
\end{align}
We now estimate $g_3(n)$. First we notice that by the maximum principle we can estimate
\[
	|h(x,4)| \leq \sup_{|z|=\frac{3}{4}} |h(z,4)| \leq \frac{4\cos\left(\frac{\pi}{6}\right) e^{\sqrt{\frac{3}{2}}\pi\frac{41}{256}}}{\frac{1}{2} + \cos\left(\frac{1}{8}\sqrt{\frac{41}{6}}\pi\right)}M(k)
\]
for all $0\leq x \leq \frac{1}{4}$, where the last inequality was shown in the proof of Proposition \ref{TaylorExpansionBound}. Hence,
\[
	\int_{0}^{\frac{1}{4}} |h(x,4)|e^{-\frac{2\pi}{3}\sqrt{n}x^2} dx \leq \frac{4\cos\left(\frac{\pi}{6}\right) e^{\sqrt{\frac{3}{2}}\pi\frac{41}{256}}}{\frac{1}{2} + \cos\left(\frac{1}{8}\sqrt{\frac{41}{6}}\pi\right)}M(k) \int_{0}^{\infty} e^{-\frac{2\pi}{3}\sqrt{n}x^2} dx.
\]
We obtain
\begin{align}\label{eqn:bound-g3}
|g_3(n)| &\leq \frac{\pi}{18\sqrt{6n}}\sqrt{\frac{2}{\pi}}\frac{E_{1}^3|a(4)|}{(\frac{2\pi}{3}\sqrt{n})^{\frac{11}{2}}}\frac{4\cos\left(\frac{\pi}{6}\right) e^{\sqrt{\frac{3}{2}}\pi\frac{41}{256}}}{\frac{1}{2} + \cos\left(\frac{1}{8}\sqrt{\frac{41}{6}}\pi\right)}M(k) \int_{0}^{\infty} e^{-\frac{2\pi}{3}\sqrt{n}x^2} dx \nonumber\\
 &= \frac{3645 e^{\frac{41}{256} \sqrt{\frac{3}{2}} \pi } \left(495+10 \left(11+\sqrt{11}\right) \sqrt{\pi }+308 \log (2)\right)}{157696
\sqrt{14} \pi ^{9/2} \left(\log (2)+\log (4) \cos \left(\frac{1}{8} \sqrt{\frac{41}{6}} \pi \right)\right)}\frac{1}{n^3} \leq \frac{1}{20}\frac{1}{n^3}
\end{align}
We proceed to estimate $g_4(n)$.  By Lemma \ref{BesselFunctionUpperBounds} and monotonicity of the Bessel function we have the upper bound
\[
\left|I_1\left(\frac{2\pi}{3}\sqrt{(1-x^2)n}\right)\right| \leq e^{\frac{5\pi}{8}\sqrt{n}}
\]
for all $n\geq 1$ and $\frac{\sqrt{31}}{16}\leq x \leq 1$. Furthermore, by a standard argument we see that for all $n \geq 30355$ the following inequality holds
\begin{equation}\label{eq:exp-terms-comparison}
\frac{e^{\frac{5\pi}{8}\sqrt{n}}}{\sqrt{n}} \leq 20 \frac{e^{\frac{2\pi}{3}\sqrt{n}}}{n^3}.
\end{equation}
Hence, for all $n \geq 30355$ we can estimate
\begin{align}\label{eqn:bound-g4}
|g_4(n)| &\leq \frac{\pi}{18\sqrt{6n}}\int_{\frac{\sqrt{31}}{16}}^{1} \left|4\frac{\cos\left(\frac{\pi}{6}\right)\cosh\left(\pi\sqrt{\frac{1}{54}}x\right)}{\cos\left(\frac{\pi}{3}\right)+\cosh\left(2\pi\sqrt{\frac{1}{54}}x\right)}\sqrt{1-x^2}I_1\left(\frac{2\pi}{3}\sqrt{(1-x^2)n}\right)\right|dx \nonumber \\
&\leq \frac{5 \left(1-\frac{\sqrt{31}}{16}\right) \pi \cosh \left(\frac{1}{48} \sqrt{\frac{31}{6}} \pi \right)}{48 \sqrt{2} \left(\frac{1}{2}+\cosh \left(\frac{1}{24} \sqrt{\frac{31}{6}} \pi \right)\right)}\frac{e^{\frac{5 \pi}{8}\sqrt{n}}}{\sqrt{n}} \leq \frac{5 \left(1-\frac{\sqrt{31}}{16}\right) \pi \cosh \left(\frac{1}{48} \sqrt{\frac{31}{6}} \pi \right)}{48 \sqrt{2} \left(\frac{1}{2}+\cosh \left(\frac{1}{24} \sqrt{\frac{31}{6}} \pi \right)\right)}\frac{e^{\frac{2 \pi}{3}\sqrt{n}}}{n^3}\leq 2 \frac{e^{\frac{2 \pi}{3}\sqrt{n}}}{n^3},
\end{align}
where we used Lemma \ref{IntegrandSimplification} and the monotonicity of the Bessel function to see that the integrands maximum in the range of integration occurs at $x=\frac{\sqrt{31}}{16}$ for the first inequality, and we used \eqref{eq:exp-terms-comparison} for the third inequality. Combining \eqref{eqn:abstract-asymptotic-formula-integral}, \eqref{eqn:bound-g1}, \eqref{eqn:bound-g2}, \eqref{eqn:bound-g3}, \eqref{eqn:bound-g4} gives
\begin{multline}\label{eqn:asymptotic-expansion-integral-explicit}
\frac{\pi}{18\sqrt{6n}}\mathcal{I}_{\frac{1}{18},1,0}(n) = e^{\frac{2\pi}{3}\sqrt{n}}\sum_{k=0}^{3} G(k) + O_{\leq}\left(\frac{11}{2}\frac{e^{\frac{2\pi}{3}\sqrt{n}}}{n^3}\right)\\
=\left(\frac{1}{18 \sqrt{2} n} -\frac{324+5 \pi ^2}{3888 \sqrt{2} \pi n^{\frac{3}{2}} } + \frac{1080+17 \pi ^2}{186624 \sqrt{2} n^2}-\frac{349920+33048 \pi ^2+455 \pi ^4}{40310784
	\sqrt{2} \pi n^{\frac{5}{2}} }\right)e^{\frac{2\pi}{3}\sqrt{n}} + O_{\leq}\left(\frac{11}{2}\frac{e^{\frac{2\pi}{3}\sqrt{n}}}{n^3}\right)
\end{multline}
for all $n\geq 31745$. We now combine \eqref{eqn:asymptotic-expansion-Bessel-function}, \eqref{eqn:asymptotic-expansion-integral-explicit} together with Lemma \ref{LehmerBound} to obtain that for all $n\geq 31745$ we have the approximation
\begin{multline*}
p_2(n) = \left(\frac{1}{4 \sqrt{3}n^{\frac{3}{4}}} + \frac{1}{18 \sqrt{2} n} -\frac{3 \sqrt{3} }{64 \pi n^{\frac{5}{4}}} -\frac{324+5 \pi
	^2}{3888 \sqrt{2} \pi n^{\frac{3}{2}} } -\frac{45 \sqrt{3} }{2048 \pi ^2 n^{\frac{7}{4}}} + \frac{1080+17 \pi ^2}{186624 \sqrt{2} n^2}-\frac{945 \sqrt{3}
	}{32768 \pi ^3 n^{\frac{9}{4}}}\right.\\
	\left.-\frac{349920+33048 \pi ^2+455 \pi ^4}{40310784 \sqrt{2} \pi n^{\frac{5}{2}} }-\frac{127575 \sqrt{3}}{2097152 \pi
	^4 n^{\frac{11}{4}}}\right)e^{\frac{2\pi}{3}\sqrt{n}}\\ + O_{\leq}\left(\frac{5893965\sqrt{3}\left(\frac{\frac{14}{5}+\frac{9}{\log(4)}}{\sqrt{2\pi}} + \frac{\sqrt{\frac{2}{11}}+\sqrt{2}}{\log(4)}\right)}{33554432\pi^5n^{\frac{13}{4}}}e^{\frac{2\pi}{3}\sqrt{n}}\right) + O_{\leq}\left(200n^{\frac{1}{16}}e^{\frac{\sqrt{3}\pi}{3}\sqrt{n}} + 46500n^{\frac{15}{16}}\right) + O_{\leq}\left(\frac{11}{2}\frac{e^{\frac{2\pi}{3}\sqrt{n}}}{n^3}\right).
\end{multline*}
We now note that
\begin{align*}
\frac{5893965\sqrt{3}\left(\frac{\frac{14}{5}+\frac{9}{\log(4)}}{\sqrt{2\pi}} + \frac{\sqrt{\frac{2}{11}}+\sqrt{2}}{\log(4)}\right)}{33554432\pi^5n^{\frac{13}{4}}}e^{\frac{2\pi}{3}\sqrt{n}} &\leq \frac{1}{100}\frac{1}{n^3}e^{\frac{2\pi}{3}\sqrt{n}},\quad \text{for } n \geq 1,\\
200n^{\frac{1}{16}}e^{\frac{\sqrt{3}\pi}{3}\sqrt{n}} &\leq 4\frac{e^{\frac{2\pi}{3}\sqrt{n}}}{n^3},\quad \text{for } n \geq 13947,\\
	46500n^{\frac{15}{16}} &\leq \frac{e^{\frac{2\pi}{3}\sqrt{n}}}{n^3}, \quad \text{for } n \geq 238.
\end{align*}

Therefore, we have shown that for $n \geq 31745$ we have
\begin{multline*}
p_2(n) = \left(\frac{1}{4 \sqrt{3}n^{\frac{3}{4}}} + \frac{1}{18 \sqrt{2} n} -\frac{3 \sqrt{3} }{64 \pi n^{\frac{5}{4}}} -\frac{324+5 \pi
	^2}{3888 \sqrt{2} \pi n^{\frac{3}{2}} } -\frac{45 \sqrt{3} }{2048 \pi ^2 n^{\frac{7}{4}}} + \frac{1080+17 \pi ^2}{186624 \sqrt{2} n^2}-\frac{945 \sqrt{3}
}{32768 \pi ^3 n^{\frac{9}{4}}}\right.\\
\left.-\frac{349920+33048 \pi ^2+455 \pi ^4}{40310784 \sqrt{2} \pi n^{\frac{5}{2}} }-\frac{127575 \sqrt{3}}{2097152 \pi
	^4 n^{\frac{11}{4}}}\right)e^{\frac{2\pi}{3}\sqrt{n}} + O_{\leq}\left(11\frac{e^{\frac{2\pi}{3}\sqrt{n}}}{n^3}\right)
\end{multline*}
If we recall the definiton of  the coefficients $a_1, \dots ,a_9$ from the beginning of the section we can write the preceeding formula as
\[
	p_2(n) = \sum_{k=1}^{9} a_k n^{-\frac{k+2}{4}}e^{\frac{2\pi}{3}\sqrt{n}} + O_{\leq}\left(11\frac{e^{\frac{2\pi}{3}\sqrt{n}}}{n^3}\right),
\]
where $n\geq 31745$. The proof is finally complete after checking the finitely many cases directly, namely we observe
\begin{equation}\label{eqn:finite-check-error-asymptotic-formula}
	\sup_{1\leq n \leq 31745} \frac{n^3\left|p_2(n) - \sum_{k=1}^{9} a_k n^{-\frac{k+2}{4}}e^{\frac{2\pi}{3}\sqrt{n}} \right|}{e^{\frac{2\pi}{3}\sqrt{n}}} \leq 15.
\end{equation}
\end{proof}
	
\section{Proof of log-concavity}
\begin{theorem}\label{LogConcavity}
		We have for $n \geq 482$  and all even $2\leq n < 482,$ that
		
		$$p_2^2(n) - p_2(n-1)p_2(n+1) > 0.$$
\end{theorem}

	\begin{proof}
		We write for $n \geq 1$ the asymptotic formula from Theorem \ref{AsymptoticFormula} in the form
		\[
		p_2(n) = \mathcal{P}(n)e^{\frac{2\pi}{3}\sqrt{n}}+O_{\leq}\left(\mathcal{E}(n)e^{\frac{2\pi}{3}\sqrt{n}}\right).
		\]
		Here we set $\mathcal{P}(n) \coloneqq \sum_{k=1}^{9} a_kn^{-\frac{k+2}{4}}$ with the coefficients $a_1, \dots, a_9$ defined at the beginning of Section $3$ and $\mathcal{E}(n) \coloneqq 15n^{-3}$. By a straight forward argument one checks that
		\begin{align*}
			\mathcal{P}(n) > 0\quad &\text{for all } n \geq 1,\\
			\mathcal{P}(n) > \mathcal{E}(n)\quad &\text{for all } n \geq 8.
		\end{align*}
		Hence, we can estimate for all $n \geq 9$
		\begin{multline*}p_2^2(n) - p_2(n-1)p_2(n+1) \\ \geq e^{\frac{4\pi}{3}\sqrt{n}}(\mathcal{P}(n)-\mathcal{E}(n))^2-e^{\frac{2\pi}{3}\sqrt{n+1}+\frac{2\pi}{3}\sqrt{n-1}}(\mathcal{P}(n+1)+\mathcal{E}(n+1))(\mathcal{P}(n-1)+\mathcal{E}(n-1)).
		\end{multline*}
		It is sufficient to show that
		\begin{equation}\label{eqn:log-concavity-start}
			(\mathcal{P}(n)-\mathcal{E}(n))^2-e^{\frac{2\pi}{3}\sqrt{n+1}+\frac{2\pi}{3}\sqrt{n-1}-\frac{4\pi}{3}\sqrt{n}}(\mathcal{P}(n+1)+\mathcal{E}(n+1))(\mathcal{P}(n-1)+\mathcal{E}(n-1)) > 0.
		\end{equation}
		
		By Taylor's Theorem we have for all $\alpha \in \R$ and $|x| \leq r$ the identity
		\begin{equation}\label{eqn:binomial-series}
			(1+x)^{\alpha} = \sum_{k=0}^{N} \binom{\alpha}{k}x^k + O_{\leq}\left(\left|\binom{\alpha}{N+1}\right|\sup_{|x|\leq r} (1+x)^{\alpha-N-1}\right),
		\end{equation}
		where $N\geq 0$ is any integer. 
	We use \eqref{eqn:binomial-series} with $\alpha = \frac{1}{2}$, $N=3$, and $r=\frac{1}{10}$ to obtain
	\[
		\sqrt{1+x} = 1 + \frac{1}{2x} - \frac{1}{8x^2} + \frac{1}{16x^3} + O_{\leq}\left(\frac{625\sqrt{\frac{5}{2}}}{17496x^4}\right).
	\]
	for all $-\frac{1}{10}\leq x \leq \frac{1}{10}$. We specialize this to $x=\frac{1}{n}$ and $x=-\frac{1}{n}$ (here we assume silently $n \geq 10$) and it follows that 
	\begin{multline*}
		\frac{2\pi}{3}\sqrt{n+1}+\frac{2\pi}{3}\sqrt{n-1}-\frac{4\pi}{3}\sqrt{n} = -\frac{1}{4n^2} \\= \frac{2\pi}{3}\sqrt{n}\left(\sqrt{1+\frac{1}{n}}+\sqrt{1-\frac{1}{n}} - 2\right)=-\frac{\pi}{6n^{\frac{3}{2}}} + O_{\leq}\left(\frac{625\pi \sqrt{\frac{5}{2}}}{13122n^{\frac{7}{2}}}\right).
	\end{multline*}
	Therefore, we can estimate \eqref{eqn:log-concavity-start} from below
	\begin{equation}\label{eqn:first-lower-bound-log-concavity}
		(\mathcal{P}(n)-\mathcal{E}(n))^2-\exp\left(-\frac{\pi}{6n^{\frac{3}{2}}}+ \frac{625\pi\sqrt{\frac{5}{2}}}{13122n^{\frac{7}{2}}}\right)(\mathcal{P}(n+1)+\mathcal{E}(n+1))(\mathcal{P}(n-1)+\mathcal{E}(n-1)).
	\end{equation}
	We have again by Taylor's Theorem
	\[
		e^x = 1 + x + O_{\leq}\left(x^2\right)
	\]
	for all $0 \leq x \leq \frac{1}{2}$. We can therefore estimate
	\begin{equation}\label{eqn:taylor-exp}
		\exp\left(-\frac{\pi}{6n^{\frac{3}{2}}}+ \frac{625\sqrt{\frac{5}{2}}}{13122n^{\frac{7}{2}}}\right) \leq 1 -\frac{\pi }{6 n^{3/2}} +\frac{\pi ^2}{36 n^3} +\frac{625 \sqrt{\frac{5}{2}} \pi }{6561 n^{7/2}} -\frac{625
			\sqrt{\frac{5}{2}} \pi ^2}{39366 n^5} +\frac{1953125 \pi ^2}{344373768 n^7}.
	\end{equation}
		We will now rewrite $\mathcal{P}(n+1)$ and $\mathcal{P}(n-1)$. From \eqref{eqn:binomial-series} one readily verifies with $r=\frac{1}{100}$ $\big($here we assume $n \geq 500 \geq 100^{\frac{4}{3}}\big)$ the following approximations
		\begingroup
		\allowdisplaybreaks
		\begin{align*}
		\frac{a_1}{\left(n\pm 1\right)^{\frac{3}{4}}} &= \frac{1}{4
		\sqrt{3}n^{\frac{3}{4}}} \mp\frac{\sqrt{3}}{16n^{\frac{7}{4}}} +\frac{7\sqrt{3}}{128n^{\frac{11}{4}}}  + O_{\leq}\left(\frac{546875 \sqrt{\frac{5}{2}}}{1587762\ 11^{3/4}n^{\frac{15}{4}}}\right),\\
		\frac{a_2}{n\pm 1} &= \frac{1}{18 \sqrt{2} n} \mp\frac{1}{18 \sqrt{2} n^2} + O_{\leq}\left(\frac{250000 \sqrt{2}}{8732691 n^3}\right),\\
		\frac{a_3}{\left(n\pm 1\right)^{\frac{5}{4}}} &= -\frac{3 \sqrt{3}}{64 \pi n^{\frac{5}{4}}} \pm \frac{15 \sqrt{3}}{256 \pi n^{\frac{9}{4}}} + O_{\leq}\left(\frac{78125 \sqrt{\frac{5}{2}}}{574992 \sqrt[4]{11} \pi n^{\frac{13}{4}}}\right),\\
		\frac{a_4}{\left(n\pm 1\right)^{\frac{3}{2}}} &= -\frac{\left(324+5 \pi ^2\right)}{3888 \sqrt{2} \pi n^{\frac{3}{2}}} \pm \frac{\left(324+5 \pi ^2\right)}{2592 \sqrt{2} \pi n^{\frac{5}{2}}} + O_{\leq}\left(\frac{390625 \left(324+5 \pi ^2\right)}{235782657 \sqrt{22} \pi n^{\frac{7}{2}}}\right),\\
		\frac{a_5}{\left(n\pm 1\right)^{\frac{7}{4}}} &= -\frac{45 \sqrt{3}}{2048 \pi ^2 n^{\frac{7}{4}}} \pm \frac{315 \sqrt{3}}{8192 \pi ^2 n^{\frac{11}{4}}} + O_{\leq}\left(\frac{2734375 \sqrt{\frac{5}{2}}}{7527168\ 11^{3/4} \pi ^2 n^{\frac{15}{4}}}\right),\\
		\frac{a_6}{\left(n\pm 1\right)^{2}} &= \frac{1080+17 \pi ^2}{186624 \sqrt{2} n^2} + O_{\leq}\left(\frac{15625 \left(1080+17 \pi ^2\right)}{1414695942 \sqrt{2} n^3}\right),\\
		\frac{a_7}{\left(n\pm 1\right)^{\frac{9}{4}}} &= -\frac{945 \sqrt{3}}{32768 \pi ^3n^{\frac{9}{4}}} + O_{\leq}\left(\frac{546875 \sqrt{\frac{5}{2}}}{4088832 \sqrt[4]{11} \pi ^3 n^{\frac{13}{4}}}\right),\\
		\frac{a_8}{\left(n\pm 1\right)^{\frac{5}{2}}} &= -\frac{\left(349920+33048 \pi ^2+455 \pi ^4\right)}{40310784 \sqrt{2} \pi n^{\frac{5}{2}}} + O_{\leq}\left(\frac{390625 \left(349920+33048 \pi ^2+455 \pi ^4\right)}{1833445940832 \sqrt{22} \pi n^{\frac{7}{2}}}\right),\\
		\frac{a_9}{\left(n\pm 1\right)^{\frac{11}{4}}} &= -\frac{127575 \sqrt{3}}{2097152 \pi ^4 n^{\frac{11}{4}}} + O_{\leq}\left(\frac{13671875 \sqrt{\frac{5}{2}}}{11894784\ 11^{3/4} \pi ^4 n^{\frac{15}{4}}}\right).\\
		\end{align*}
		\endgroup
			Summing these terms up we obtain for $n \geq 500$ the following asymptotic formula
		
		\begin{multline*}
		\mathcal{P}(n \pm 1) + O_{\leq}\left(\mathcal{E}(n\pm 1)\right) = \frac{1}{4 \sqrt{3}n^{\frac{3}{4}}} + \frac{1}{18 \sqrt{2}n} -\frac{3 \sqrt{3}}{64 \pi n^{\frac{5}{4}}} -\frac{324+5 \pi ^2}{3888 \sqrt{2} \pi n^{\frac{3}{2}}} + \frac{\sqrt{3} \left(\mp 128 \pi ^2-45\right)}{2048 \pi ^2 n^{\frac{7}{4}}}\\
		+ \frac{1080 \mp 10368 + 17 \pi ^2}{186624 \sqrt{2} n^2} + \frac{15 \sqrt{3} \left(\pm 128 \pi ^2-63\right)}{32768 \pi ^3 n^{\frac{9}{4}}} - \frac{349920 \mp 5038848 + (33048 \mp 77760) \pi ^2+455 \pi ^4}{40310784 \sqrt{2} \pi n^{\frac{5}{2}}}\\
		+ \frac{7 \sqrt{3} \left(-18225 \pm 11520 \pi ^2+16384 \pi ^4\right)}{2097152 \pi ^4 n^{\frac{11}{4}}} + O_{\leq}\left(\frac{16}{n^3}\right),
		\end{multline*}
		where we used that for $n \geq 5$ we can write
		\begin{align}
		\mathcal{E}(n-1) &\leq \frac{15.15}{n^3}.
		\end{align}
		In particular, we have for $n \geq 5$
		\begin{multline}\label{eqn:shift-back-to-n}
		\mathcal{P}(n \pm 1) + \mathcal{E}(n\pm 1) \leq \frac{1}{4 \sqrt{3}n^{\frac{3}{4}}} + \frac{1}{18 \sqrt{2}n} -\frac{3 \sqrt{3}}{64 \pi n^{\frac{5}{4}}} -\frac{324+5 \pi ^2}{3888 \sqrt{2} \pi n^{\frac{3}{2}}} + \frac{\sqrt{3} \left(\mp 128 \pi ^2-45\right)}{2048 \pi ^2 n^{\frac{7}{4}}}\\
		+ \frac{1080 \mp 10368 + 17 \pi ^2}{186624 \sqrt{2} n^2} + \frac{15 \sqrt{3} \left(\pm 128 \pi ^2-63\right)}{32768 \pi ^3 n^{\frac{9}{4}}} - \frac{349920 \mp 5038848 + (33048 \mp 77760) \pi ^2+455 \pi ^4}{40310784 \sqrt{2} \pi n^{\frac{5}{2}}}\\ +
		\frac{7 \sqrt{3} \left(-18225 \pm 11520 \pi ^2+16384 \pi ^4\right)}{2097152 \pi ^4 n^{\frac{11}{4}}} + \frac{16}{n^3},
		\end{multline}
		By a standard argument one shows that the right hand side of \eqref{eqn:shift-back-to-n} is positive for all $n\geq 1$. Inserting \eqref{eqn:taylor-exp} and \eqref{eqn:shift-back-to-n} into \eqref{eqn:first-lower-bound-log-concavity} gives for $n \geq 500$ with constants $c_k$ given in the Appendix (Section 6)
		\begin{equation}\label{eqn:lower-bound-log-concavity-2}
			p_2^2(n) - p_2(n-1)p_2(n+1) \geq \frac{\pi }{288n^3} +\frac{\pi }{216 \sqrt{6}n^{\frac{13}{4}}} + \sum_{k=0}^{38} \frac{c_k}{n^{\frac{7}{2}+\frac{k}{4}}}.
		\end{equation}
	By forgetting all positive $c_k$ in \eqref{eqn:lower-bound-log-concavity-2} we infer
	\[
	p_2^2(n) - p_2(n-1)p_2(n+1) \geq \frac{\pi }{288n^3} +\frac{\pi }{216 \sqrt{6}n^{\frac{13}{4}}} - \frac{1}{50n^{\frac{7}{2}}} - \frac{9}{n^{\frac{15}{4}}} - \frac{5}{2n^4} - \frac{200}{n^5} > 0
	\]
	for $n \geq 7667$. The finitely many remaining cases are checked directly. This completes the proof.
	\end{proof}
	\begin{remark}
		As mentioned in Remark \ref{rmk:asymptotic-formula},  Bringmann and Mahlburg obtained in \cite{BringmannMahlburg} the asymptotic formula
		$$p_2(n) \sim \left(\frac{1}{4\sqrt{3}n^\frac{3}{4}}+\frac{1}{18\sqrt{2}n}\right)e^{\frac{2\pi}{3}\sqrt{n}}.$$
		However, this formula is not strong enough to prove log-concavity with exactly our method. Indeed, if one expands the log-concavity condition with abstract coefficients one can show that an asymptotic formula with precision to at least order $n^{-\frac{9}{4}}$ is necessary for our method to work properly. We give the details here. Let $A_1, \dots, A_6$ be real numbers and let $M>0$ be positive and assume that
		\begin{align*}
		\mathcal{P}(n) &= \sum_{k=0}^{6} A_kn^{-\frac{k+2}{4}},\\
		\mathcal{E}(n) &= \frac{M}{n^{\frac{9}{4}}}.
		\end{align*}
		If one now follows the preceeding proof with those definitions instead, then \eqref{eqn:lower-bound-log-concavity-2} would have leading term $\frac{-96\sqrt{3}M+\pi}{288n^3}$ which would require a really small value of $M$ to work. If one takes even less terms in the asymptotic expansion, the leading term in \eqref{eqn:lower-bound-log-concavity-2} is by the same argument always strictly negative.
		
		We chose to expand up to order $n^{-\frac{11}{4}}$ to potentially reduce the number of finitely many cases we have to check directly.
	\end{remark}
	
	\section{Higher Turán inequalities}
	In this section we are going to prove the Turán inequalities for $p_2(n)$. More precisely we are going to show that for all $d\geq 1$ and $n$ sufficiently large the Jensen polynomials
	\[
	J_{p_2(n)}^{d,n}(X) = \sum_{j=0}^{d} \binom{d}{j}p_2(n+j)X^j
	\]
	are hyperbolic, i.e., have only real zeros. We will use a general criterion of Griffin, Ono, Rolen and Zagier \cite{GriffinOnoRolenZagier}. 
	\begin{theorem}[Theorem 3 and Corollary in \cite{GriffinOnoRolenZagier}]\label{Griffin-Ono-Rolen-Zagier criterion}
		Let $\{\alpha(n)\}$, $\{A(n)\}$ and $\{\delta(n)\}$ be three sequences of positive real numbers with $\delta(n)$ tending to zero. Let $d\geq 3$ be an integer. Let $C_3(n), \dots C_d(n)$ be real valued functions satisfying $C_i(n) = O\left(\delta(n)^i\right)$ for all $3 \leq j \leq d$. Suppose that for all $0 \leq j \leq d$ we have the asymptotic formula
		\[
		\log\left(\frac{\alpha(n+j)}{\alpha(n)}\right) = A(n)j -\delta(n)^2j^2 + \sum_{i=3}^{d} C_i(n)j^i + o\left(\delta(n)^d\right), \quad \text{as } n \rightarrow \infty.
		\]
		Then, the polynomial $J_{\alpha}^{d,n}(X)$ is hyperbolic for all sufficiently large $n$.
	\end{theorem}
	\begin{remark}
		\begin{itemize}
			\item [(i)] The theorem in the original reference \cite{GriffinOnoRolenZagier} has a typo. The sum $\sum_{i=3}^{d} C_i(n)j^i$ is missing. It is mentioned correctly in Theorem 6 in \cite{GriffinOnoRolenZagier}, which implies the above Theorem directly. The same version as above can be found in \cite{OnoPujahariRolen}.
			\item [(ii)] As remarked in \cite{OnoPujahariRolen} we state that the above criterion can also be stated for $d\in \{1,2\}$. However, we only need the result for $d \geq 3$ as the proof of Theorem \ref{Higher-Turan-Inequalities} shows.
		\end{itemize}
	\end{remark}
	
	This criterion has been applied to a large class of sequences. We want to give two examples. The first one has been given directly in the aforementioned article \cite{GriffinOnoRolenZagier}. Following the authors we denote by $f$ any modular form with real Fourier coefficients on the full modular group $\SL_2(\Z)$ that is holomorphic except from a pole of positive order at infinity. Any such $f$ has a Fourier expansion
	\[
	f(\tau) = \sum_{n \in m + \Z_{\geq 0}} a_f(n)q^n,
	\]
	where $m$ is a positive rational number and $a_f(-m) \neq 0$. The authors proved \cite[Theorem 5]{GriffinOnoRolenZagier}, that for such a form and any fixed $d\geq 1$ the Jensen polynomials $J_{a_f}^{d,n}$ are hyperbolic for all sufficiently large $n$. This class is quite rich and in particular includes the modular form $f=\frac{1}{\eta}$ whose Fourier coefficients are essentially the partitions numbers $p(n)$. One concludes from this (see \cite{GriffinOnoRolenZagier}) that for any fixed $d \geq 1$ the Jensen polynomials $J_{p}^{d,n}$ are hyperbolic for all but finitely many $n\geq 1$. Larson and Wagner \cite{LarsonWagner} made this result explicit in the following sense. For each fixed $d \geq 1$ they gave an integer $N(d) \geq 1$ such that for all $n \geq N(d)$ the Jensen polynomial $J_{p}^{d,n}$ is hyperbolic.
	
	The second example comes from the plane partition function. For any integer $n \geq 0$ we define the number of plane partitions $\pl(n)$ of $n$ by the generating series identity
	\[
	\sum_{n=0}^{\infty} \pl(n) q^n \coloneqq \prod_{n=1}^{\infty} \frac{1}{(1-q^n)^n}.
	\]
	Ono, Pujahari, and Rolen were able to establish strong asymptotic formulas for $\pl(n)$ with explicit error term \cite[Theorem 1.3]{OnoPujahariRolen}. They used their results to show that $\pl(n)$ is log-concave for $n \geq 12$, answering a conjecture of Heim, Neuhauser, and Tröger \cite{HeimNeuhauserTroeger}. Additionally the authors went beyond log-conacvity and established that $J_{\pl}^{d,n}$ is hyperbolic for all $d \geq 1$ and $n$ sufficiently large \cite[Theorem 1.2]{OnoPujahariRolen}, answering another conjecture they credited to Heim, Neuhauser, and Tröger \cite{OnoPujahariRolen}. Pandey made this result explicit \cite{Pandey} by providing for each fixed $d \geq 1$ an integer $N(d)$ such that for all $n \geq N(d)$ the polynomial $J_{\pl}^{d,n}$ is hyperbolic.
	\begin{theorem}\label{Higher-Turan-Inequalities}
		Let $d \geq 1$ be an integer. Then, for all but finitely many values of $n$ the Jensen polynomial $J_{p_2}^{d,n}$ is hyperbolic, i.e. has only real roots.
	\end{theorem}
	\begin{proof}
		The case $d=1$ is clear. Let $d=2$. Then, we ask for the roots of the polynomial
		\[
		J_{p_2}^{2,n} = p_2(n+2)X^2+2p_2(n+1)X+p_2(n).
		\]
		By the quadratic formula its zeros are real precisely, if the discrimant is non-negative, i.e.
		\[
		4\cdot p_2(n+1)^2-4\cdot p_2(n+2)p_2(n) \geq 0.
		\]
		Shifting $n \mapsto n-1$ in the previous expression shows that $J_{p_2}^{2,n-1}$ has only real roots if and only if $p_2$ is log-concave at $n$.
		Therefore, we now that $J_{p_2}^{2,n}$ is hyperbolic for all $n \geq 481$ and all odd $1\leq n< 481$. For the Higher-Turán inequalities, we will use Theorem \ref{Griffin-Ono-Rolen-Zagier criterion}. Our proof is analogous to \cite[Theorem 7]{GriffinOnoRolenZagier} and \cite[Theorem 1.2]{OnoPujahariRolen}. Let $d \geq 3$. Following the proof of Theorem \ref{AsymptoticFormula} it is clear that there is a divergent expansion
		\[
		p_2(n) \sim \frac{e^{\frac{2\pi}{3}\sqrt{n}}}{n^{\frac{3}{4}}} \sum_{k=0}^{\infty} a_kn^{-\frac{k}{4}}
		\]
		accurate up to all orders of $n$ where the coefficients $a_k$ are all real numbers. Using the Taylor expansion $\log(1-x) = -\sum_{k=1}^{\infty} \frac{x^k}{k}$ we can rewrite the asymptotic behaviour as
		\[
		p_2(n) \sim e^{\frac{2\pi}{3}\sqrt{n}}n^{-\frac{3}{4}}\cdot \exp\left(c_0 + \frac{c_1}{n^{\frac{1}{4}}}+ \frac{c_2}{n^{\frac{1}{2}}} + \frac{c_3}{n^{\frac{3}{4}}} + \cdots \right),
		\]
		for some real numbers $c_k$ accurate up to all orders of $n$. Therefore, we have for all $0\leq j \leq d$ the asymptotics
		\[
		\log\left(\frac{p_2(n+j)}{p_2(n)}\right) \sim \frac{2\pi}{3} \sum_{i=1}^{\infty} \binom{1/2}{i} \frac{j^i}{n^{i-\frac{1}{2}}} - \frac{3}{4} \sum_{i=1}^{\infty} \frac{(-1)^{i-1}j^i}{in^i} + \sum_{\ell,k \geq 1} c_\ell \binom{-\ell/4}{s} \frac{j^k}{n^{k+\frac{\ell}{4}}},
		\]
		as $n \rightarrow \infty$. Hence, $\{p_2(n)\}$ satisfies the conditions to apply Theorem \ref{Griffin-Ono-Rolen-Zagier criterion} with $A(n) = \frac{\pi}{3\sqrt{n}} + O\left(\frac{1}{n}\right)$, and $\delta(n) = \sqrt{\pi/12}\cdot n^{-\frac{3}{4}} + O\left(n^{-\frac{5}{4}}\right).$
	\end{proof}
	Here we give the values of $c_k$ for equation \eqref{eqn:lower-bound-log-concavity-2}. These coefficients appear as the result of expanding out the explicit log-concavity inequality coming from Inserting \eqref{eqn:taylor-exp}, and \eqref{eqn:shift-back-to-n} into \eqref{eqn:first-lower-bound-log-concavity}. The following table provides the exact values of $c_k$ together with an numerical approximation. For the numerical approximation we always rounded the third digit up to provide an upper bound.
	\vspace{2cm}
	\begin{center}
		\renewcommand{\arraystretch}{1.75}
		\begin{tabular}{|T|L|T|}
			\hline
			\frac{7}{2} + \frac{k}{2} & c_k & \approx\\\hline
			\hline
			\frac{7}{2} & \frac{\pi }{3888}-\frac{5}{256} & -0.019 \\\hline
			\frac{15}{4} & -\frac{31}{2 \sqrt{3}}-\frac{1}{144 \sqrt{6}}-\frac{5 \pi ^2}{46656 \sqrt{6}} & -8.953 \\\hline
			4 & \frac{1}{1296}-\frac{31}{9 \sqrt{2}}-\frac{15}{2048 \pi }-\frac{5 \pi ^2}{419904} & -2.438 \\\hline
			\frac{17}{4} & \frac{93 \sqrt{3}}{32 \pi }-\frac{149}{4096 \sqrt{6} \pi }+\frac{5 \pi }{20736 \sqrt{6}}+\frac{17 \pi
				^3}{2239488 \sqrt{6}} & 1.598 \\\hline
			\frac{9}{2} & \frac{225}{131072 \pi ^2}-\frac{11}{1728 \pi }+\frac{31}{6 \sqrt{2} \pi }-\frac{5 \pi }{139968}+\frac{155 \pi
			}{1944 \sqrt{2}}-\frac{\pi ^2}{1728}+\frac{89 \pi ^3}{90699264} & 1.333 \\\hline
			\frac{19}{4} & -\frac{215}{98304 \sqrt{6}}+\frac{685 \sqrt{\frac{3}{2}}}{65536 \pi ^2}+\frac{1395 \sqrt{3}}{1024 \pi
				^2}-\frac{145 \pi ^2}{165888 \sqrt{6}}-\frac{455 \pi ^4}{483729408 \sqrt{6}} & 0.236 \\\hline
			5 & \frac{5}{62208}-\frac{155}{432 \sqrt{2}}+\frac{40905}{2097152 \pi ^3}+\frac{1}{128 \pi ^2}+\frac{5 \pi }{1536}-\frac{625
				\sqrt{\frac{5}{2}} \pi }{629856}-\frac{515 \pi ^2}{10077696}-\frac{527 \pi ^2}{93312 \sqrt{2}}-\frac{5 \pi ^4}{40310784} &
			-0.287 \\\hline
			\frac{21}{4} & \frac{38955 \sqrt{\frac{3}{2}}}{4194304 \pi ^3}+\frac{29295 \sqrt{3}}{16384 \pi ^3}+\frac{17035}{4718592
				\sqrt{6} \pi }-\frac{625 \sqrt{\frac{5}{3}} \pi }{944784}+\frac{4 \pi }{3 \sqrt{3}}+\frac{62123 \pi }{42467328
				\sqrt{6}}+\frac{2375 \pi ^3}{107495424 \sqrt{6}} & 2.519 \\\hline
							\end{tabular}
		\end{center}
	\begin{center}
\renewcommand{\arraystretch}{1.75}
\begin{tabular}{|T|L|T|}
\hline
			\frac{11}{2} & -\frac{127}{16384}+\frac{625 \sqrt{\frac{5}{2}}}{559872}+\frac{467775}{16777216 \pi ^4}+\frac{5}{41472 \pi
			}+\frac{155}{288 \sqrt{2} \pi }-\frac{3025 \pi }{26873856}+\frac{3599 \pi }{10368 \sqrt{2}}-\frac{625 \sqrt{\frac{5}{2}} \pi
			}{8503056}+\frac{1745 \pi ^3}{725594112}+\frac{14105 \pi ^3}{20155392 \sqrt{2}}+\frac{11701 \pi ^5}{3761479876608} & 0.902 \\\hline
			\frac{23}{4} & \frac{6875 \sqrt{\frac{5}{3}}}{5038848}-2 \sqrt{3}+\frac{1376669}{226492416 \sqrt{6}}+\frac{14175
				\sqrt{\frac{3}{2}}}{8388608 \pi ^4}+\frac{3954825 \sqrt{3}}{1048576 \pi ^4}+\frac{3425}{33554432 \sqrt{6} \pi ^2}+\frac{3125
				\sqrt{\frac{5}{3}} \pi ^2}{204073344}-\frac{366365 \pi ^2}{9172942848 \sqrt{6}}-\frac{17 \pi ^4}{13436928 \sqrt{6}} & -3.393
			\\\hline
			6 & -\frac{92562361}{2985984}-\frac{2 \sqrt{2}}{9}+\frac{625 \sqrt{\frac{5}{2}}}{2834352}+\frac{7526925}{4294967296 \pi
				^5}-\frac{75}{262144 \pi }+\frac{625 \sqrt{\frac{5}{2}}}{2985984 \pi }+\frac{3245 \pi ^2}{644972544}-\frac{5 \pi ^2}{729
				\sqrt{2}}+\frac{3125 \sqrt{\frac{5}{2}} \pi ^2}{918330048}-\frac{9763 \pi ^4}{52242776064}-\frac{7735 \pi ^6}{45137758519296}
			& -31.361 \\\hline
			\frac{25}{4} & -\frac{2275 \sqrt{\frac{3}{2}}}{67108864 \pi ^3}-\frac{6875 \sqrt{\frac{5}{3}}}{17915904 \pi }-\frac{15
				\sqrt{3}}{128 \pi }-\frac{8388125}{1610612736 \sqrt{6} \pi }-\frac{3125 \sqrt{\frac{5}{3}} \pi }{40310784}+\frac{17848885 \pi
			}{48922361856 \sqrt{6}}-\frac{10625 \sqrt{\frac{5}{3}} \pi ^3}{9795520512}+\frac{17 \pi ^3}{995328 \sqrt{6}}+\frac{455 \pi
				^5}{2902376448 \sqrt{6}} & -0.066 \\\hline
			\frac{13}{2} & \frac{120558375}{68719476736 \pi ^6}-\frac{13635}{4194304 \pi ^2}+\frac{3125 \sqrt{\frac{5}{2}}}{7962624 \pi
				^2}-\frac{5159}{3981312 \pi }-\frac{625 \sqrt{\frac{5}{2}}}{3779136 \pi }-\frac{875 \pi }{71663616}+\frac{5 \pi }{162
				\sqrt{2}}-\frac{3125 \sqrt{\frac{5}{2}} \pi }{153055008}-\frac{\pi ^2}{2304}+\frac{625 \sqrt{\frac{5}{2}} \pi
				^2}{1889568}+\frac{100655 \pi ^3}{69657034752}+\frac{17 \pi ^3}{34992 \sqrt{2}}-\frac{55625 \sqrt{\frac{5}{2}} \pi
				^3}{198359290368}+\frac{111415 \pi ^5}{5015306502144}+\frac{207025 \pi ^7}{19499511680335872} & 0.08 \\\hline
			\frac{27}{4} & \frac{59375 \sqrt{\frac{5}{3}}}{429981696}-\frac{69695735}{115964116992 \sqrt{6}}+\frac{23625
				\sqrt{\frac{3}{2}}}{134217728 \pi ^4}-\frac{822115 \sqrt{\frac{3}{2}}}{536870912 \pi ^2}-\frac{3125
				\sqrt{\frac{5}{3}}}{95551488 \pi ^2}-\frac{315 \sqrt{3}}{2048 \pi ^2}+\frac{4020625 \sqrt{\frac{5}{3}} \pi
				^2}{17414258688}-\frac{2 \pi ^2}{9 \sqrt{3}}+\frac{5461 \pi ^2}{254803968 \sqrt{6}}+\frac{284375 \sqrt{\frac{5}{3}} \pi
				^4}{2115832430592}-\frac{455 \pi ^4}{644972544 \sqrt{6}} & -1.291 \\\hline
			7 & -\frac{5}{248832}-\frac{5}{108 \sqrt{2}}+\frac{3125 \sqrt{\frac{5}{2}}}{68024448}+\frac{16275380625}{8796093022208 \pi
				^7}-\frac{155925}{33554432 \pi ^3}+\frac{3125 \sqrt{\frac{5}{2}}}{3145728 \pi ^3}+\frac{43 \pi }{32768}-\frac{625
				\sqrt{\frac{5}{2}} \pi }{559872}+\frac{6481 \pi ^2}{161243136}-\frac{209 \pi ^2}{3888 \sqrt{2}}+\frac{2414375
				\sqrt{\frac{5}{2}} \pi ^2}{88159684608}-\frac{305 \pi ^4}{4353564672}-\frac{455 \pi ^4}{7558272 \sqrt{2}}+\frac{3125
				\sqrt{\frac{5}{2}} \pi ^4}{88159684608}-\frac{11701 \pi ^6}{22568879259648} & -0.413 \\\hline
			\frac{29}{4} & -\frac{4725 \sqrt{\frac{3}{2}}}{16777216 \pi ^3}+\frac{284375 \sqrt{\frac{5}{3}}}{2038431744 \pi ^3}-\frac{42525
				\sqrt{3}}{131072 \pi ^3}-\frac{934375 \sqrt{\frac{5}{3}}}{20639121408 \pi }-\frac{3425}{201326592 \sqrt{6} \pi
			}-\frac{73926875 \sqrt{\frac{5}{3}} \pi }{185752092672}+\frac{\pi }{\sqrt{3}}-\frac{1250 \sqrt{\frac{10}{3}} \pi
			}{6561}-\frac{1478045 \pi }{1358954496 \sqrt{6}}-\frac{19484375 \sqrt{\frac{5}{3}} \pi ^3}{3761479876608}-\frac{463075 \pi
				^3}{55037657088 \sqrt{6}} & 0.7 \\\hline
			\frac{15}{2} & \frac{45}{524288}-\frac{6875 \sqrt{\frac{5}{2}}}{17915904}-\frac{2508975}{8589934592 \pi ^4}-\frac{1203125
				\sqrt{\frac{5}{2}}}{1207959552 \pi ^4}-\frac{3125 \sqrt{\frac{5}{2}}}{90699264 \pi }+\frac{764408473 \pi
			}{17915904}+\frac{\sqrt{2} \pi }{27}-\frac{269375 \sqrt{\frac{5}{2}} \pi }{58773123072}-\frac{2500 \sqrt{5} \pi
			}{59049}-\frac{10925 \pi ^3}{3869835264}+\frac{5 \pi ^3}{4374 \sqrt{2}}-\frac{1990625 \sqrt{\frac{5}{2}} \pi
				^3}{1586874322944}+\frac{1219 \pi ^5}{313456656384}-\frac{7313125 \sqrt{\frac{5}{2}} \pi ^5}{8226356490141696}+\frac{7735 \pi
				^7}{270826551115776} & 133.933 \\\hline
			\frac{31}{4} & \frac{625 \sqrt{\frac{5}{6}}}{2916}-\frac{797058125 \sqrt{\frac{5}{3}}}{990677827584}+\frac{5
				\sqrt{3}}{256}+\frac{8326685}{9663676416 \sqrt{6}}-\frac{109375 \sqrt{\frac{5}{3}}}{150994944 \pi ^4}-\frac{2140625
				\sqrt{\frac{5}{3}}}{146767085568 \pi ^2}+\frac{2275}{134217728 \sqrt{6} \pi ^2}+\frac{747378125 \sqrt{\frac{5}{3}} \pi
				^2}{40122452017152}-\frac{9969205 \pi ^2}{293534171136 \sqrt{6}}+\frac{10625 \sqrt{\frac{5}{3}} \pi
				^4}{29386561536}-\frac{119 \pi ^4}{143327232 \sqrt{6}} & 0.229 \\\hline
					\end{tabular}
		\end{center}
	\begin{center}
\renewcommand{\arraystretch}{1.75}
\begin{tabular}{|T|L|T|}
\hline
			8 & \frac{5159}{23887872}-\frac{1784375 \sqrt{\frac{5}{2}}}{6530347008}+\frac{1250
				\sqrt{5}}{19683}-\frac{40186125}{137438953472 \pi ^5}-\frac{6453125 \sqrt{\frac{5}{2}}}{12884901888 \pi
				^5}+\frac{5355}{8388608 \pi }+\frac{3125 \sqrt{\frac{5}{2}}}{191102976 \pi }+\frac{2795 \pi ^2}{429981696}-\frac{5 \pi
				^2}{972 \sqrt{2}}+\frac{2771875 \sqrt{\frac{5}{2}} \pi ^2}{1410554953728}+\frac{3125 \sqrt{5} \pi ^2}{3188646}+\frac{16561
				\pi ^4}{417942208512}-\frac{17 \pi ^4}{209952 \sqrt{2}}+\frac{11441875 \sqrt{\frac{5}{2}} \pi ^4}{114254951251968}-\frac{7735
				\pi ^6}{30091839012864}+\frac{4834375 \sqrt{\frac{5}{2}} \pi ^6}{98716277881700352}-\frac{207025 \pi ^8}{116997070082015232}
			& 0.123 \\\hline
			\frac{33}{4} & -\frac{7875 \sqrt{\frac{3}{2}}}{268435456 \pi ^3}+\frac{1421875 \sqrt{\frac{5}{3}}}{97844723712 \pi
				^3}+\frac{3125 \sqrt{\frac{5}{6}}}{31104 \pi }+\frac{283745 \sqrt{\frac{3}{2}}}{1073741824 \pi }+\frac{5280978125
				\sqrt{\frac{5}{3}}}{7044820107264 \pi }+\frac{105 \sqrt{3}}{4096 \pi }-\frac{16080353125 \sqrt{\frac{5}{3}} \pi
			}{213986410758144}+\frac{63572215 \pi }{695784701952 \sqrt{6}}-\frac{435625 \sqrt{\frac{5}{3}} \pi
				^3}{104485552128}+\frac{119 \pi ^3}{15925248 \sqrt{6}}-\frac{284375 \sqrt{\frac{5}{3}} \pi ^5}{6347497291776}+\frac{3185 \pi
				^5}{30958682112 \sqrt{6}} & 0.044 \\\hline
			\frac{17}{2} & -\frac{5425126875}{17592186044416 \pi ^6}-\frac{34453125 \sqrt{\frac{5}{2}}}{68719476736 \pi
				^6}+\frac{363825}{536870912 \pi ^2}+\frac{259375 \sqrt{\frac{5}{2}}}{339738624 \pi ^2}+\frac{3224375
				\sqrt{\frac{5}{2}}}{8707129344 \pi }+\frac{5 \pi }{648 \sqrt{2}}-\frac{653125 \sqrt{\frac{5}{2}} \pi
			}{156728328192}-\frac{3125 \sqrt{5} \pi }{708588}-\frac{3109289329 \pi ^2}{8463329722368}+\frac{625 \sqrt{\frac{5}{2}} \pi
				^2}{2519424}+\frac{17 \pi ^3}{23328 \sqrt{2}}-\frac{136169375 \sqrt{\frac{5}{2}} \pi ^3}{152339935002624}-\frac{10625
				\sqrt{5} \pi ^3}{153055008}+\frac{455 \pi ^5}{45349632 \sqrt{2}}-\frac{134434375 \sqrt{\frac{5}{2}} \pi
				^5}{10968475320188928}-\frac{129390625 \sqrt{\frac{5}{2}} \pi ^7}{42645432044894552064} & 0.001 \\\hline
			\frac{35}{4} & \frac{15795678125 \sqrt{\frac{5}{3}}}{169075682574336}-\frac{546875 \sqrt{\frac{5}{3}}}{7247757312 \pi
				^4}+\frac{21875 \sqrt{\frac{5}{6}}}{165888 \pi ^2}+\frac{495621875 \sqrt{\frac{5}{3}}}{782757789696 \pi ^2}+\frac{14175
				\sqrt{3}}{262144 \pi ^2}-\frac{6983125 \sqrt{\frac{5}{3}} \pi ^2}{557256278016}-\frac{7 \pi ^2}{48 \sqrt{3}}+\frac{1250
				\sqrt{\frac{10}{3}} \pi ^2}{19683}-\frac{1953125 \pi ^2}{12397455648 \sqrt{6}}+\frac{284375 \sqrt{\frac{5}{3}} \pi
				^4}{2507653251072} & 0.335 \\\hline
			9 & \frac{3125 \sqrt{\frac{5}{2}}}{272097792}+\frac{3125 \sqrt{5}}{472392}-\frac{4651171875 \sqrt{\frac{5}{2}}}{8796093022208
				\pi ^7}+\frac{1203125 \sqrt{\frac{5}{2}}}{805306368 \pi ^3}+\frac{1953125 \pi }{14693280768}-\frac{23125 \sqrt{\frac{5}{2}}
				\pi }{71663616}-\frac{1586876276069 \pi ^2}{223154201664}-\frac{4050625 \sqrt{\frac{5}{2}} \pi ^2}{176319369216}+\frac{250625
				\sqrt{5} \pi ^2}{17006112}+\frac{190625 \sqrt{\frac{5}{2}} \pi ^4}{4760622968832}+\frac{284375 \sqrt{5} \pi
				^4}{33059881728}+\frac{7313125 \sqrt{\frac{5}{2}} \pi ^6}{24679069470425088} & -69.844 \\\hline
			\frac{37}{4} & \frac{109375 \sqrt{\frac{5}{6}}}{393216 \pi ^3}+\frac{109375 \sqrt{\frac{5}{3}}}{452984832 \pi ^3}+\frac{2140625
				\sqrt{\frac{5}{3}}}{440301256704 \pi }-\frac{625 \sqrt{\frac{5}{6}} \pi }{1944}+\frac{923778125 \sqrt{\frac{5}{3}} \pi
			}{2972033482752}+\frac{21484375 \pi }{66119763456 \sqrt{6}}+\frac{289421875 \sqrt{\frac{5}{3}} \pi
				^3}{120367356051456}+\frac{9765625 \pi ^3}{2677850419968 \sqrt{6}} & -0.913 \\\hline
			\frac{19}{2} & \frac{1953125}{78364164096}-\frac{3125 \sqrt{\frac{5}{2}}}{63700992}+\frac{6453125
				\sqrt{\frac{5}{2}}}{38654705664 \pi ^4}+\frac{1953125 \pi }{74384733888}-\frac{238876575625 \sqrt{\frac{5}{2}} \pi
			}{19591041024}-\frac{1250 \sqrt{5} \pi }{59049}+\frac{9765625 \pi ^3}{24100653779712}+\frac{6828125 \sqrt{\frac{5}{2}} \pi
				^3}{4231664861184}-\frac{3125 \sqrt{5} \pi ^3}{9565938}-\frac{761875 \sqrt{\frac{5}{2}} \pi
				^5}{342764853755904}-\frac{4834375 \sqrt{\frac{5}{2}} \pi ^7}{296148833645101056} & -60.739 \\\hline
			\frac{39}{4} & -\frac{3125 \sqrt{\frac{5}{6}}}{93312}-\frac{5204178125
				\sqrt{\frac{5}{3}}}{21134460321792}-\frac{21484375}{235092492288 \sqrt{6}}-\frac{1421875 \sqrt{\frac{5}{3}}}{293534171136 \pi
				^2}+\frac{6230753125 \sqrt{\frac{5}{3}} \pi ^2}{641959232274432}-\frac{9765625 \pi ^2}{528958107648 \sqrt{6}}+\frac{74375
				\sqrt{\frac{5}{3}} \pi ^4}{313456656384}-\frac{33203125 \pi ^4}{128536820158464 \sqrt{6}} & -0.031 \\\hline
					\end{tabular}
		\end{center}
	\begin{center}
\renewcommand{\arraystretch}{1.75}
\begin{tabular}{|T|L|T|}
\hline
			10 & -\frac{1953125}{99179645184}-\frac{3224375 \sqrt{\frac{5}{2}}}{26121388032}+\frac{11484375 \sqrt{\frac{5}{2}}}{68719476736
				\pi ^5}+\frac{9765625}{208971104256 \pi }-\frac{371875 \sqrt{\frac{5}{2}}}{1019215872 \pi }-\frac{9765625 \pi
				^2}{4016775629952}-\frac{1746875 \sqrt{\frac{5}{2}} \pi ^2}{470184984576}+\frac{3125 \sqrt{5} \pi
				^2}{2125764}-\frac{173828125 \pi ^4}{5205741216417792}-\frac{10350625 \sqrt{\frac{5}{2}} \pi ^4}{457019805007872}+\frac{10625
				\sqrt{5} \pi ^4}{459165024}+\frac{4834375 \sqrt{\frac{5}{2}} \pi ^6}{32905425960566784}+\frac{129390625 \sqrt{\frac{5}{2}}
				\pi ^8}{127936296134683656192} & 0.038 \\\hline
			\frac{41}{4} & \frac{546875 \sqrt{\frac{5}{3}}}{21743271936 \pi ^3}-\frac{21875 \sqrt{\frac{5}{6}}}{497664 \pi
			}-\frac{177340625 \sqrt{\frac{5}{3}}}{782757789696 \pi }-\frac{9765625}{1253826625536 \sqrt{6} \pi }-\frac{39732634375
				\sqrt{\frac{5}{3}} \pi }{1521681143169024}+\frac{185546875 \pi }{5642219814912 \sqrt{6}}-\frac{74375 \sqrt{\frac{5}{3}} \pi
				^3}{34828517376}+\frac{564453125 \pi ^3}{228509902503936 \sqrt{6}}-\frac{1990625 \sqrt{\frac{5}{3}} \pi
				^5}{67706637778944}+\frac{888671875 \pi ^5}{27763953154228224 \sqrt{6}} & -0.013 \\\hline
			\frac{21}{2} & \frac{1550390625 \sqrt{\frac{5}{2}}}{8796093022208 \pi ^6}+\frac{9765625}{82556485632 \pi ^2}-\frac{8421875
				\sqrt{\frac{5}{2}}}{21743271936 \pi ^2}+\frac{9765625 \pi }{1785233613312}-\frac{3125 \sqrt{5} \pi }{1417176}-\frac{1953125
				\pi ^2}{22039921152}+\frac{30625 \sqrt{\frac{5}{2}} \pi ^2}{214990848}+\frac{794921875 \pi ^3}{2313662762852352}-\frac{10625
				\sqrt{5} \pi ^3}{51018336}+\frac{9765625 \pi ^5}{2313662762852352}-\frac{284375 \sqrt{5} \pi ^5}{99179645184} & -0.031 \\\hline
			\frac{43}{4} & -\frac{2919921875}{270826551115776 \sqrt{6}}-\frac{109375 \sqrt{\frac{5}{6}}}{1179648 \pi
				^2}+\frac{888671875}{26748301344768 \sqrt{6} \pi ^2}+\frac{4375 \sqrt{\frac{5}{6}} \pi ^2}{52488}-\frac{1953125 \pi
				^2}{43046721 \sqrt{3}}+\frac{32978515625 \pi ^2}{2437438960041984 \sqrt{6}}-\frac{888671875 \pi ^4}{49358138940850176
				\sqrt{6}} & 0.484 \\\hline
			11 & -\frac{9765625}{2380311484416}-\frac{3759765625}{31701690482688 \pi ^3}-\frac{1953125 \pi }{52242776064}+\frac{12658203125
				\pi ^2}{1542441841901568}-\frac{1953125 \sqrt{2} \pi ^2}{387420489}+\frac{40000 \sqrt{10} \pi ^2}{19683}-\frac{595703125 \pi
				^4}{41645929731342336}-\frac{22853515625 \pi ^6}{215892499727278669824} & 63.356 \\\hline
			\frac{45}{4} & -\frac{341796875}{1981355655168 \sqrt{6} \pi ^3}-\frac{6689453125}{1925877696823296 \sqrt{6} \pi }+\frac{1953125
				\pi }{76527504 \sqrt{3}}-\frac{2886806640625 \pi }{12999674453557248 \sqrt{6}}-\frac{904443359375 \pi ^3}{526486815369068544
				\sqrt{6}} & 0.046 \\\hline
			\frac{23}{2} & \frac{9765625}{557256278016}-\frac{20166015625}{338151365148672 \pi ^4}-\frac{6701171875 \pi
			}{171382426877952}+\frac{1953125 \pi }{129140163 \sqrt{2}}-\frac{21337890625 \pi ^3}{37018604205637632}+\frac{9765625 \pi
				^3}{41841412812 \sqrt{2}}+\frac{2380859375 \pi ^5}{2998506940656648192}+\frac{15107421875 \pi ^7}{2590709996727344037888} &
			0.039 \\\hline
			\frac{47}{4} & \frac{9765625}{816293376 \sqrt{3}}+\frac{16263056640625}{92442129447518208
				\sqrt{6}}+\frac{4443359375}{1283918464548864 \sqrt{6} \pi ^2}-\frac{19471103515625 \pi ^2}{2807929681968365568
				\sqrt{6}}-\frac{232421875 \pi ^4}{1371059415023616 \sqrt{6}} & 0.007 \\\hline
			12 & \frac{10076171875}{228509902503936}-\frac{11962890625}{200385994162176 \pi ^5}+\frac{1162109375}{8916100448256 \pi
			}+\frac{5458984375 \pi ^2}{4113178245070848}-\frac{9765625 \pi ^2}{9298091736 \sqrt{2}}+\frac{32345703125 \pi
				^4}{3998009254208864256}-\frac{33203125 \pi ^4}{2008387814976 \sqrt{2}}-\frac{15107421875 \pi
				^6}{287856666303038226432}-\frac{404345703125 \pi ^8}{1119186718586212624367616} & -0.009 \\\hline
			\frac{49}{4} & -\frac{1708984375}{95105071448064 \sqrt{6} \pi ^3}+\frac{68359375}{4353564672 \sqrt{3} \pi
			}+\frac{554189453125}{3423782572130304 \sqrt{6} \pi }+\frac{124164482421875 \pi }{6655833320221310976
				\sqrt{6}}+\frac{232421875 \pi ^3}{152339935002624 \sqrt{6}}+\frac{6220703125 \pi ^5}{296148833645101056 \sqrt{6}} & 0.003 \\\hline
			\frac{25}{2} & -\frac{59814453125}{949978046398464 \pi ^6}+\frac{26318359375}{190210142896128 \pi ^2}+\frac{9765625 \pi
			}{6198727824 \sqrt{2}}-\frac{95703125 \pi ^2}{1880739938304}+\frac{33203125 \pi ^3}{223154201664 \sqrt{2}}+\frac{888671875
				\pi ^5}{433811768034816 \sqrt{2}} & 0.007 \\\hline
			\frac{51}{4} & \frac{341796875}{10319560704 \sqrt{3} \pi ^2}-\frac{13671875 \pi ^2}{459165024 \sqrt{3}} & -0.168 \\\hline
			13 & -\frac{62500000 \pi ^2}{43046721} & -14.33 \\\hline
		\end{tabular}
	\end{center}

\end{document}